\documentclass{svjour2}                    % onecolumn
\smartqed  % flush right qed marks, e.g. at end of proof
\usepackage{graphicx}
\usepackage{amssymb}
\begin{document}

\title{Quasi-periodic relativistic strings in the Minkowski space $\textbf{R}^{1+n}$%\thanks{Grants or other notes
%about the article that should go on the front page should be
%placed here. General acknowledgments should be placed at the end of the article.}
}
%\subtitle{Do you have a subtitle?\\ If so, write it here}

%\titlerunning{Short form of title}        % if too long for running head

\author{Weiping Yan $^{a}$, Binlin Zhang $^{b}$
         %etc.
}

%\authorrunning{Short form of author list} % if too long for running head

\institute{
\footnotesize $^a$ School of Mathematical Sciences, Xiamen University, Xiamen, 361005, China.\\
\email{yanwp@xmu.edu.cn}\\
\footnotesize $^b$ College of Mathematics and System Science, Shandong University of Science and Technology,
Qingdao, 266590, China\\
\email{zhangbinlin2012@163.com}}
\date{Received: date / Accepted: date}
% The correct dates will be entered by the editor

\maketitle

\begin{abstract}
In this article we consider the motion of relativistic strings in the Minkowski space $\textbf{R}^{1+n}$.
Those surfaces are known as a timelike minimal surface, and described by a system with $n$
nonlinear wave equations of Born-Infeld type. 
By constructing a suitable Nash-Moser
iteration scheme, we prove that the $n$-dimensional relativistic strings can admit a more generalized time quasi-periodic motion in $\textbf{R}^{1+n}$.
Moreover, those time quasi-periodic solutions are also timelike solutions.

%\keywords{ \and }
% \PACS{PACS code1 \and PACS code2 \and more}
\end{abstract}

\section{Introduction and Main Result}
\label{intro}
Relativistic strings arise in the context of membrane, supermembrane theories and higher-dimensional extensions of string theory, where they are called $p$-branes according to the dimension of the spacelike object. This equation is also related to the Born-Infeld theory which appear in string theory and relativity theory. This triggers the revival of interests in the original Born-Infeld electromagnetism \cite{Born} and the exploration of Born-Infeld gauge theory \cite{Gibbon}.
From the mathematical point of view, this theory is a nonlinear generalization of the
Maxwell theory. Brenier \cite{Bel} even carried out a
study of the theory in the connection to hydrodynamics. On the other hand, this theory is
also related to the theory for the $\textbf{C}^2$ timelike extremal surfaces, i.e.
 timelike submanifolds with vanishing mean curvature.

In this paper, we consider the following $n$-nonlinear wave equations of Born-Infeld type
\begin{eqnarray}\label{E1-3}
|x_{\theta}|^2x_{tt}-2\langle x_t,x_{\theta}\rangle x_{t\theta}+(|x_t|^2-1)x_{\theta\theta}=0,
\end{eqnarray}
where $x=(x_1,\cdots,x_n)$ denotes an embedding from $\textbf{R}\times\textbf{T}$ to $\textbf{R}^{n}$, $x_k=x_k(t,\theta)$ is a suitable coordinate system in parameter form for $k=1,\cdots,n$ and $(t,\theta)\in\textbf{R}\times\textbf{T}$, $\textbf{T}$ denotes a one dimensional torus.

Equation (\ref{E1-3}) shows that the surface is extremal if and only if its mean curvature vector vanishes. Although in the process of deriving (\ref{E1-3}) we assume that the surface is timelike, there equations themselves do not need this assumption. One can see \cite{Kong2,Kong1} for more detail on the derivation of equation (\ref{E1-3}).

Membrane equations are one case of an important class of geometric wave equations which is the Lorentzian analogue of the minimal submanifold equations. Since such equations possess plenty of geometric phenomenon and complicated structure, (for example, they develop singularities in finite time, and degenerate and not strictly hyperbolic properties.) much work is attracted in recent years. The case of non-compact timelike maximal graphs in Minkowski spacetimes $\textbf{R}^{1+n}$ is fairly
well understood. Global well-posedness for sufficiently small initial data was established by
Brendle \cite{Brendle} and by Lindblad \cite{Lindblad}. The case of general codimension and local well-posedness in the light cone gauge were studied by Allen,
Andersson-Isenberg \cite{Allen} and Allen-Andersson-Restuccia \cite{Allen1}, respectively. Yan-Zhang \cite{YZ} showed that the bosonic membrane equation in the light cone gauge is local well-posedness on time interval $[0,\varepsilon^{-\frac{1}{2}}T]$ even if the initial Riemannian metric may be degenerate, where $\varepsilon$ is the small parameter measures the nonlinear effects.
Kong and his collaborators \cite{Kong2} showed that equation
(\ref{E1-3}) has a global $\textbf{C}^2$-solution and presented many numerical evidence where topology singularity formation is prominent, and they \cite{Kong1}
obtained a representation formula of solution and took a time-periodic motion. Bellettini, Hoppe, Novaga and Orlandi \cite{Bel} showed that if the initial curve is a centrally symmetric convex curve and the initial velocity is zero, the string shrinks to a point in
finite time. They noticed that it should be noted that the string does not become extinct there, but rather
comes out of the singularity point, evolves back to its original shape and then periodically
afterwards. Nguyen and Tian \cite{Ngu} showed that timelike maximal cylinders in orthogonal gauge in $\textbf{R}^{1+2}$ always develop topological singularities and that, infinitesimally at a generic singularity, their time slices are evolved
by a rigid motion or a self-similar motion. They also proved there exists a blow up in
non-flat backgrounds. Recently, Yan \cite{Yan0} showed that the one dimensional Born-Infeld equation admits a family of explicit stable self-similar blowup solutions.

In this paper, we show that the timelike minimal surface admits time the quasi-periodic motion in Minkowski space $\textbf{R}^{1+n}$.
More precisely, we obtain the following main result:
\begin{theorem}
The timelke minimal surface $\Sigma$ can have a time quasi-periodic motion in Minkowski space $\textbf{R}^{1+n}$, i.e. there exists a classical solution
\begin{eqnarray*}
x(t,\theta)=(t+\theta+u_1(t+\theta,\omega t),\cdots,t+\theta+u_n(t+\theta,\omega t))^T,
\end{eqnarray*}
which is time quasi-periodic with frequency $(1,\omega)$ and periodic in $\theta$ with periodic $T$, and satisfies (\ref{E1-3}) with the timelike character
\begin{eqnarray}\label{E1-R1*}
\langle x_t,x_{\theta}\rangle^2-(|x_t|^2-1)|x_{\theta}|^2>0,
\end{eqnarray}
where $\omega\in\mathcal{D}_{\gamma,\tau}$ satisfies
\begin{eqnarray}\label{E1-4R}
|n\omega^2l^2-j^2|\geq\frac{\gamma}{|l|^{\tau}},~~(l,j)\in\textbf{Z}^2/\{(0,0)\},
\end{eqnarray}
with $0<\gamma<1$ and $\tau>0$. Here $n\geq1$ denotes the dimension of Minkowski space, $\mathcal{D}_{\gamma,\tau}\subset[0,\epsilon_0]\times[\omega_1,\omega_2]$ denotes a Cantor like set of Lebesgue measure $|\mathcal{D}_{\gamma,\tau}|\geq\epsilon_0(|\omega_2-\omega_1|-C\gamma)$.

Furthermore, let $\epsilon_0>0$, $0<\sigma_{0}<\bar{\sigma}$. The Sobolev regularity solution of (\ref{E1-3}) with the timelike character gives rise to a global and uniqueness map $x(\epsilon,\omega)\in\textbf{C}^1([0,\epsilon_0]\times[\omega_1,\omega_2];\textbf{H}_{\bar{\sigma}})$ with $\|u(\epsilon,\omega)\|_{\bar{\sigma}}=O(\epsilon),~~\|\partial_{\epsilon,\omega}u(\epsilon,\omega)\|_{\bar{\sigma}}\leq C\gamma^{-1}$, where $C$ denotes a positive constant, $\omega$ denotes the frequency parameter, $\epsilon$ denotes the amplitude parameter.
\end{theorem}

We notice that above result takes a small initial data, i.e. for the small initial
data $x(0,\theta)=x_0=\theta+\epsilon u_0(\theta)$, above result holds. Specially, for a closed string, it can take time quasi-periodic motion for $(t,\theta)\in\textbf{T}^2$. We mention that Kong and Zhang \cite{Kong1} proved the existence of time periodic solution of the timelike minimal surface equation (\ref{E1-3}). Meanwhile, they gave a representation formula of solution. Due to loss of ``good'' structure of equation (see (\ref{E1-5})) and the ``small divisor'' problem, for showing the time quasi-periodic solution, the representation formula of solution can not use. We also notice that timelike maximal cylinders in orthogonal gauge in $\textbf{R}^{1+2}$ always develop singularities in finite time (see \cite{Ngu}), but in other gauge, this timelike minimal surface equation can have a global smooth solution. For example, Kong \cite{Kong2} showed that equation (\ref{E1-3}) has a global $\textbf{C}^2$-solution.

We remark that the form of quasi-periodic solution
\begin{eqnarray}\label{E1-4}
x_k(t,\theta)=t+\theta+u_k(t+\theta,\omega t),~~k=1,\cdots,n,
\end{eqnarray}
 depends on the structure of the equation (\ref{E1-3}). More precisely, in order to obtain the wave operator $\partial_{tt}-\partial_{xx}$, direct computation shows that the solution need to be translated by $t+\theta$. In other word, let $t+\theta=\vartheta$, the target of this paper is to prove that equation (\ref{E1-3}) has a quasi-periodic solution $x(\vartheta,\omega t)=(\vartheta+u_1(\vartheta,\omega t),\cdots,\vartheta+u_n(\vartheta,\omega t))^{T}\in\textbf{T}^n$. This is motivated by standard KAM theory (see for example \cite{Ar}) which implies that for sufficient small $\epsilon>0$, the ODE
\begin{eqnarray*}
x''+\epsilon f_x(t,x)=0
\end{eqnarray*}
admits many quasi-periodic solutions, that
is, solutions of the form
\begin{eqnarray*}
x(t)=t+u(t,\omega t),
\end{eqnarray*}
where $\omega$ is diophantine and belongs to the set of Diophantine numbers in $\textbf{R}$, $f(t,x)$ denotes a external force and satisfies some assumption on $x$.

Assumption
(\ref{E1-4R}) means that the forcing frequency $\omega$ does not enter in resonance with the normal
mode frequencies $\omega_j:=|j|$ of oscillations of the membrane. By standard arguments, (\ref{E1-4R}) is satisfied for all $\omega\in[\omega_1,\omega_2]$ but a subset of measure $O(\gamma)$, for fixed $0<\omega_1<\omega_2$. Since the nonresonance condition (\ref{E1-4R}) is the same with \cite{Berti2}, so we omit the proof process of measure estimate.

Inserting (\ref{E1-4}) into (\ref{E1-3}), thus the motion for the displacement is governed by a nonlinear wave system
\begin{eqnarray}\label{E1-5}
n\omega^2u_{tt}-u_{yy}+\omega^2\left(2\sum_{k=1}^nu_{ky}+|u_y|^2\right)u_{tt}-2\omega^2\left(\sum_{k=1}^nu_{kt}+\langle u_t, u_y\rangle\right)u_{ty}+\omega^2|u_t|^2u_{yy}=0,~~~
\end{eqnarray}
where $u=(u_1,\cdots,u_n)$, $(t,y)\in\textbf{T}^2$ and $k=1,\cdots,n$.

Let us denote by $\textbf{H}_{s}:=H_s\underbrace{\times\cdots\times}_{n}H_s$ the real Sobolev space with the norm $\|U\|_{s}=\sum_{k=1}^n\|u_k\|_s$, for all $U=(u_1,\cdots,u_n)\in\textbf{H}_{s}$, where
\begin{eqnarray*}
H_s&:=&H_s(\textbf{T}^2;\textbf{R})\nonumber\\
&:=&\left\{u(t,x)=\sum_{(l,j)\in\textbf{Z}^2}u_{l,j}e^{i(lt+jy)},~~u_{l,j}^*=u_{-l,-j}|\|u\|_s^2:=\sum_{(l,j)\in\textbf{Z}^2}e^{2(|l|+|j|)s}|u_{l,j}|^2<+\infty\right\},
\end{eqnarray*}
and $s>1$. $\textbf{H}_{s}$ and $H_s$ are Banach algebra with respect to the multiplication of
functions, namely
\begin{eqnarray*}
u_1,u_2\in H_{s}\Longrightarrow u_1u_2\in H_{s}~~and~~\|u_1u_2\|_{s}\leq\|u_1\|_{s}\|u_2\|_{s}.
\end{eqnarray*}

In fact, we reduce our problem into finding quasi-periodic solution like (\ref{E1-4}) of $n$-nonlinear wave equations of
Born-Infeld type (\ref{E1-3}). This orbit of time-quasi-periodic solution forms a torus $\mathcal{T}$.
The problem of finding periodic solutions for nonlinear PDEs has attracted much
attention which dates back to the work of Rabinowitz \cite{Rabi1}. He studied the existence of time-periodic
solutions with a rational frequency for a one-dimensional nonlinear wave equation
by variational approach. But for other frequencies of time-periodic solution, variational
approach seems weaker, due to the appearance of the so-called ``small divisor''. The first result to overcome this difficulty in solving one dimensional nonlinear wave equation were constructed by Kuksin \cite{Kuksin1}. Later on,  Craig and Wayne \cite{Craig} treated the same problem with the periodic
boundary condition for time-periodic solutions by using the analytic Newton iteration scheme combining the Lyapunov-Schmidt decomposition. The main difficulty of this strategy lies in the Green's function analysis and the control of the inverse of infinite matrices with small eigenvalues. Bourgain \cite{Bourgain1,Bourgain2} extended the Craig-Wayne's method to obtain
the existence of quasi-periodic solutions of nonlinear wave equations and Schr\"{o}dinger equations. Berti and Bolle \cite{Berti2} generalized previous results of Bourgain to more general nonlinearities of class of $\textbf{C}^k$ and assumed weaker non-resonance conditions. We notice that above results deal with nonlinear PDE with Hamiltonian structure, but nonlinear wave system (\ref{E1-5}) has not Hamiltonian structure. Thus, inspired by the work of \cite{PRR,R,Yan,YZ}, a suitable Nash-Moser iteration scheme needs to be constructed.

Instead of looking for solutions of (\ref{E1-5}) in a shrinking neighborhood of zero, it is a convenient device to perform the rescaling
\begin{eqnarray*}
u\longrightarrow\epsilon u,~~\epsilon>0,
\end{eqnarray*}
having
\begin{eqnarray}\label{E1-6}
n\omega^2u_{tt}-u_{yy}+\epsilon\omega^2\left(2\sum_{k=1}^nu_{ky}+\epsilon|u_y|^2\right)u_{tt}-2\epsilon\omega^2\left(\sum_{k=1}^nu_{kt}+\epsilon\langle u_t, u_y\rangle\right)u_{ty}+\epsilon^2\omega^2|u_t|^2u_{yy}=0.~~~~~
\end{eqnarray}
We perform the Lyapunov-Schmidt reduction via the orthogonal decomposition
\begin{eqnarray*}
\textbf{H}_{s}=\textbf{H}_{s}^0\oplus\textbf{R},
\end{eqnarray*}
where $\textbf{H}_s^0$ denotes the Sobolev functions with zero mean value.

Then the solution of (\ref{E1-3}) take the form
\begin{eqnarray}\label{E1-10}
u(t,\theta)=w(t,\theta)+M,
\end{eqnarray}
where $w(t,\theta)\in\textbf{H}_{s}^0$ and $M$ is a constant.

Inserting (\ref{E1-10}) into (\ref{E1-6}), the problem is reduced to
\begin{eqnarray}\label{E1-7}
n\omega^2w_{tt}-w_{yy}+\epsilon\omega^2\left(2\sum_{k=1}^nw_{ky}+\epsilon|w_y|^2\right)w_{tt}-2\epsilon\omega^2\left(\sum_{k=1}^nw_{kt}+\epsilon\langle w_t, w_y\rangle\right)w_{ty}+\epsilon^2\omega^2|w_t|^2w_{yy}=0.~~~~
\end{eqnarray}

Assume that equation (\ref{E1-3}) has a time quasi-periodic solution $x(t,\theta)=(t+\theta+u_1(t+\theta,\omega t),\cdots,t+\theta+u_n(t+\theta,\omega t))^T$, we shows that the timelike character (\ref{E1-R1*}) holds. Note that the solution in Theorem 1 is a small solution and the rescaling $u\longrightarrow\epsilon u$. By direction computation and (\ref{E1-R1*}), for sufficient small $\epsilon>0$, we have
\begin{eqnarray*}
\langle x_t,x_{\theta}\rangle^2&-&(|x_t|^2-1)|x_{\theta}|^2=\left(\sum_{k=1}^n(1+2\epsilon u_{ky}+\omega\epsilon u_{kt}+\epsilon^2u_{ky}^2+\omega\epsilon^2 u_{kt}u_{ky})\right)^2\\
&&-\left(\sum_{k=1}^n(1+2\epsilon u_{ky}+2\omega\epsilon u_{kt}+2\omega\epsilon^2u_{ky}u_{kt}+\epsilon^2u_{ky}^2+\omega^2\epsilon^2u_{kt}^2)-1\right)\cdot\left(\sum_{k=1}^n(1+2\epsilon u_{ky}+\epsilon^2u_{ky}^2)\right)\\
&=&n+\epsilon R(u)>0,
\end{eqnarray*}
where $R(u)$ denotes the terms of $u$'s derivatives and $n\geq 1$ is the dimension of Minkowski space. Thus, for small amplitude time quasi-periodic solution of equation (\ref{E1-3}), the timelike character (\ref{E1-R1*}) holds.

The structure of the paper is as follows: In next section, we give a crucial analysis of
the inverse of the linearized operators, which plays a crucial role in the Nash-Moser iteration.
The last section is devoted to construct a Nash-Moser iteration scheme to solve the nonlinear equation (\ref{E1-7}).

\begin{acknowledgements}
The first author expresses his sincerely thanks to the BICMR of Peking University and Professor Gang Tian for constant support and encouragement.
The first author is supported by NSFC No 11771359, and the Fundamental Research Funds for the Central Universities (Grant No. 20720190070, No.201709000061 and No. 20720180009). The second author is supported by NSFC (No. 11871199).
\end{acknowledgements}

\section{Analysis of the Linearized operator}

This section will carry out the analysis of the inverse of linearized operator. We consider the orthogonal splitting
\begin{eqnarray*}
\textbf{H}_s:=\textbf{W}^{(N_p)}\oplus\textbf{W}^{(N_p)\bot},
\end{eqnarray*}
where
\begin{eqnarray*}
\textbf{W}^{(N_p)}=\{w=(w_1,\cdots,w_n)\in\textbf{H}_s|w_k=\sum_{|(l,j)|\leq N_p}w_{k,l,j}e^{i(lt+jy)},~~k=1,\cdots,n\}
\end{eqnarray*}
and
\begin{eqnarray*}
\textbf{W}^{(N_p)\bot}=\{w=(w_1,\cdots,w_n)\in\textbf{H}_s|w_k=\sum_{|(l,j)|\geq N_p}w_{k,l,j}e^{i(lt+jy)},~~k=1,\cdots,n\},
\end{eqnarray*}
where $p$ denotes the ``$p$''th iterative step. For a given suitable $N_0>1$, we take $N_p=N_0^p$, $\forall p\in\textbf{N}$.

Define the projectors
$\Psi^{(N_p)}:\textbf{H}_s\longrightarrow\textbf{W}^{(N_p)}$. For $\forall N>0,~\forall s,~d\geq0$, it satisfies the ``smoothing'' properties:
\begin{eqnarray}
&&\|\Psi^{(N)}w\|_{s+d}\leq e^{N_p^d}\|w\|_s,~~\forall w\in \textbf{H}_s,\nonumber\\
\label{E3-1R2}
&&\|(I-\Psi^{(N)})w\|_{s}\leq N_p^{-d}\|w\|_{s+d},~~\forall w\in \textbf{H}_{s+d}.
\end{eqnarray}
Consider the truncation equation of (\ref{E1-7})
\begin{eqnarray}\label{E3-6}
\mathcal{J}(w):=J_{\omega}w+2\omega^2\epsilon\Psi^{(N_p)}f(w_t,w_y,w_{tt},w_{yy})=0,
\end{eqnarray}
where
\begin{eqnarray*}
J_{\omega}w=n\omega^2w_{tt}-w_{yy}
\end{eqnarray*}
and
\begin{eqnarray}\label{E3-1}
f(w_t,w_y,w_{tt},w_{yy})=\sum_{k=1}^n(w_{ky}+\frac{\epsilon}{2}|w_y|^2)w_{tt}
-\sum_{k=1}^n(w_{kt}+\epsilon\langle w_t, w_y\rangle)w_{ty}+\frac{\epsilon}{2}|w_t|^2w_{yy}.~~~~~
\end{eqnarray}
By direct computation, the linearized operator of (\ref{E3-6}) has the following form
\begin{eqnarray}\label{E3-3}
\mathcal{J}_{\omega}^{(N_p)}:=\Psi^{(N_p)}(\mathcal{J}_{\omega}+2\omega^2\epsilon D_{w}f(w_t,w_y,w_{tt},w_{yy}))|_{\textbf{W}^{(N_p)}},
\end{eqnarray}
where
\begin{eqnarray}\label{E3-3R1}
D_{w}f(w_t,w_y,w_{tt},w_{yy})h&=&\sum_{k=1}^n(w_{ky}+\frac{\epsilon}{2}|w_y|^2)h_{tt}-\sum_{k=1}^n(w_{ky}+\epsilon\langle w_t,w_y\rangle)h_{ty}\nonumber\\
&&+\frac{\epsilon}{2}|w_t|^2h_{yy}+\sum_{k=1}^nh_{ky}w_{tt}+2\epsilon\langle w_y,h_y\rangle w_{tt}\nonumber\\
&&-\sum_{k=1}^nh_{kt}w_{ty}-\epsilon(\langle w_t,h_y\rangle+\langle h_t,w_y\rangle) w_{ty}+\epsilon\langle w_t,h_t\rangle w_{yy}.~~~~~~
\end{eqnarray}
\begin{lemma}
For any $s>0$, there exists a positive constants $C_{\epsilon}$ such that
\begin{eqnarray}\label{E3-3R2}
\|\Psi^{(N_p)}D_wf(w_t,w_y,w_{tt},w_{yy})h\|_{s}\leq C_{\epsilon} N_p^4\left(\|w\|_{s}^2\|h\|_{s}+2\|w\|_{s}\|h\|_{s}\right),~~~~
\end{eqnarray}
\begin{eqnarray}
\label{E3-3R3}
\|\Psi^{(N_p)}(f(w_t+h_t,w_y+h_y,w_{tt}+h_{tt},w_{yy}+h_{yy})&-&f(w_t,w_y,w_{tt},w_{yy})-D_wf(w_t,w_y,w_{tt},w_{yy})h)\|_{s}\nonumber\\
&\leq& C_{\epsilon}N_p^4(C_3+\|w\|_{s})(\|h\|_{s}^2+\|h\|_{s}^3),
\end{eqnarray}
\begin{eqnarray}
\label{E3-3R4}
\|\Psi^{(N_p)}f(w_t,w_y,w_{tt},w_{yy})\|_{s}\leq C_{\epsilon}N_p^4\|w\|_{s}^3.
\end{eqnarray}
\end{lemma}
\begin{proof}
By direct computation, we have
\begin{eqnarray}\label{E2-A}
\|\Psi^{(N_p)}\sum_{k=1}^nw_{ky}h_{tt}\|_{s}&\leq&\sum_{k=1}^n\|\Psi^{(N_p)}(w_{ky}h_{tt})\|_{s}\nonumber\\
&=&\sum_{k=1}^n\|\Psi^{(N_p)}(\sum_{(l,j)\in\textbf{Z}^2}j(iw_{k,l,j})e^{i(lt+jy)})(\sum_{(l,j)\in\textbf{Z}^2}l^2h_{l,j}e^{i(lt+jy)})\|_{s}\nonumber\\
%&\leq&\sum_{k=1}^n\|\sum_{k=1}^n\sum_{|(l,j)|\leq N_p}j(iw_{k,l,j})e^{i(lt+jy)}\|_{s}\|\sum_{|(l,j)|\leq N_p}l^2h_{k,l,j}e^{i(lt+jy)}\|_{s}\nonumber\\
&\leq&N_p^3\sum_{k=1}^n\left(\sum_{(l,j)\in\textbf{Z}^2}|w_{k,l,j}|e^{s(|l|+|j|)}\right)\left(\sum_{(l,j)\in\textbf{Z}^2}|h_{l,j}|e^{s(|l|+|j|)}\right)\nonumber\\
&\leq&N_p^3\|w\|_{s}\|h\|_{s}.
\end{eqnarray}
%and
%\begin{eqnarray}\label{E2-AA}
%\|\Psi^{(N_p)}|w_y|^2h_{tt}\|_{s}&=&\sum_{k=1}^n\|\Psi^{(N_p)}|w_y|^2h_{ktt}\|_{s}\leq\sum_{k=1}^n\|\Psi^{(N_p)}|v_y+w_y|^2\|_{s}\|\Psi^{(N_p)}h_{ktt}\|_{s}\nonumber\\
%&\leq&\sum_{k=1}^n\left(\Psi^{(N_p)}\|\langle v_y,v_y\rangle\|_s+2\|\Psi^{(N_p)}\langle v_y, w_y\rangle\|_s+\|\Psi^{(N_p)}\langle w_y,w_y\rangle\|_s\right)\|\Psi^{(N_p)}h_{ktt}\|_{s}\nonumber\\
%&\leq&N_p^4\sum_{k=1}^n\left(\|v\|_{s}^2+2\|v\|_{s}\|w\|_{s}+\|w\|_{s}^2\right)\|h_{k}\|_{s}\nonumber\\
%&\leq&N_p^4\left(\|v\|_{s}^2+2\|v\|_{s}\|w\|_{s}+\|w\|_{s}^2\right)\|h\|_{s}.
%\end{eqnarray}
Using the similar computation method, we obtain
\begin{eqnarray*}
&&\|\Psi^{(N_p)}\sum_{k=1}^nw_{ky}h_{ty}\|_{s}\leq N_p^3\|w\|_{s}\|h\|_{s},\\
&&\|\Psi^{(N_p)}\langle w_t,w_y\rangle h_{ty}\|_{s},~\||w_t|^2h_{yy}\|_{s}\leq N_p^4\|w\|_{s}^2\|h\|_{s},\\
&&\|\Psi^{(N_p)}\sum_{k=1}^nh_{ky}w_{tt}\|_{s},~~\|\Psi^{(N_p)}\sum_{k=1}^nh_{kt}w_{ty}\|_{s}\leq N_p^3\|h\|_{s}\|w\|_{s},\\
&&\|\Psi^{(N_p)}\langle w_y,h_y\rangle w_{tt}\|_{s},~\|\Psi^{(N_p)}\langle w_t,h_y\rangle w_{ty})\|_{s}\leq N_p^4\|w\|_{s}^2\|h\|_{s},\\
&&\|\Psi^{(N_p)}\langle h_t,w_y\rangle w_{ty}\|_{s},~\|\Psi^{(N_p)}\langle w_t,h_t\rangle w_{yy}\|_{s}\leq N_p^4\|w\|_{s}^2\|h\|_{s}.
\end{eqnarray*}
Above estimate combining with the Young inequality give that
\begin{eqnarray*}
\|\Psi^{(N_p)}D_wf(w_t,w_y,w_{tt},w_{yy})h\|_s\leq C_{\epsilon}N_p^4\left(\|w\|_{s}^2\|h\|_{s}+2\|w\|_{s}\|h\|_{s}\right),
\end{eqnarray*}
where $C_{\epsilon}$ denotes a constant depending on $\epsilon$.

Next we prove (\ref{E3-3R3}) and (\ref{E3-3R4}). By (\ref{E3-1}) and (\ref{E3-3R1}), we derive
\begin{eqnarray*}
f(w_t&+&h_t,w_y+h_y,w_{tt}+h_{tt},w_{yy}+h_{yy})-f(w_t,w_y,w_{tt},w_{yy})-D_wf(w_t,w_y,w_{tt},w_{yy})h\nonumber\\
&=&(\sum_{k=1}^nh_{ky})h_{tt}+\frac{\epsilon}{2}(2\langle w_y,h_y\rangle+|h_y|^2)h_{tt}+\epsilon|h_y|^2w_{tt}+(\sum_{k=1}^nh_{kt})h_{ty}\nonumber\\
&&+\epsilon\left((\langle w_t,h_y\rangle+\langle w_y,h_t\rangle)h_{ty}+\langle h_t,h_y\rangle w_{ty}\right)+2\epsilon(\langle w_t,h_t\rangle+|h_t|^2) h_{yy}.
\end{eqnarray*}
By the similar estimate with (\ref{E2-A}), we obtain
\begin{eqnarray*}
\|\Psi^{(N_p)}(f(w_t&+&h_t,w_y+h_y,w_{tt}+h_{tt},w_{yy}+h_{yy})-f(w_t,w_y,w_{tt},w_{yy})-D_wf(w_t,w_y,w_{tt},w_{yy})h)\|_{s}\nonumber\\
&\leq& C_{\epsilon}N_p^4(C_3+\|w\|_{s})(\|h\|_{s}^2+\|h\|_{s}^3).
\end{eqnarray*}
The estimate (\ref{E3-3R4}) is similar, so we omit it. This completes the proof.
\end{proof}
Now we give the main result in this section.
\begin{lemma}
Assume that
\begin{eqnarray}\label{E10-1}
|\omega^2q-p|\geq\frac{\gamma}{\max\{1,|p|^{\mu}\}},~~\forall(q,p)\in\textbf{Z}^2/\{(0,0)\},~\forall\mu>1.
\end{eqnarray}
and $\|w\|_{\bar{\sigma}}\leq 1$, $\forall$ $1\leq r\leq N_p$,
\begin{eqnarray}\label{E8-1}
\|(\mathcal{J}^{(r)}_{\omega}(\epsilon,w))^{-1}\|_0\leq\frac{r^{\tau}}{\gamma_1}.
\end{eqnarray}
Then the linearized operator $\mathcal{J}^{(N_p)}_{\omega}(\epsilon,w)$ is invertible and
$\forall$ $s_2>s_1>\bar{\sigma}>0$, the linearized operator $\mathcal{J}_{\omega}^{(N_p)}$ satisfies
\begin{eqnarray}\label{E3-4}
\|(\mathcal{J}^{{(N_p)}}_{\omega}(\epsilon,w))^{-1}h\|_{s_1}\leq C(s_2-s_1)N_p^{\tau+\kappa_0}\left(1+\epsilon cN_p^{4+s_2}\|w\|_{s_2}\right)^3\|h\|_{s_2},~~~
\end{eqnarray}
where $C(s_2-s_1)=c(s_2-s_1)^{-\tau}$, $c=c_{\varsigma,\tau,s_1,s_2,\gamma_1,\gamma}$ denotes a constant.
\end{lemma}

In what follows, we carry out proving Lemma 2. Let
\begin{eqnarray*}
b(t,x):=(\partial_wf)(w_t(t,x),w_y(t,x),w_{tt}(t,x),w_{yy}(t,x)).
\end{eqnarray*}
For fixing $\varsigma>0$, we define the regular sites $R$ and the singular sites $S$ as
\begin{eqnarray}\label{E4-9}
R:=\{a\in\Omega_{N_p}||d_{(l,j)}|\geq\varsigma\}~~and~~S:=\{a\in\Omega_{N_p}||d_{(l,j)}<\varsigma\}.
\end{eqnarray}
Since the orthogonal decomposition $\textbf{H}^{(N_p)}=\textbf{H}_R\bigoplus\textbf{H}_S$,
we identify a linear operator $A$ acting on $\textbf{H}_s$ with its matrix representation $A=(A_b^{a})_{a,b\in\Omega_{N_p}}$ with blocks $A_b^{a}\in\mathcal{J}(\mathcal{N}_{a},\mathcal{N}_{b})$. We define the polynomially localized block matrices
\begin{eqnarray*}
\mathcal{A}_s:=\{A=(A_b^{a})_{a,b\in\Omega_{N_p}}:|A|^2_s:=\sup_{b\in\Omega_{N_p}}\sum_{a\in\Omega_{N_p}}e^{2s|b-a|}\|A_b^{a}\|_0^2<\infty\},
\end{eqnarray*}
where $\Omega_{N_p}=\{(l,j)\in\textbf{Z}^2||(l,j)|\leq N_p\}$, $\|A_b^{a}\|_0:=\sup_{w\in\mathcal{N}_{a},\|w\|_0=1}\|A_b^{a}w\|_0$ is the $\textbf{L}^2$-operator norm in $\mathcal{J}(\mathcal{N}_{a},\mathcal{N}_{b})$, for $\mathcal{N}_a,\mathcal{N}_b\subset\Omega_{N_p}$. If $s'>s$, then these holds $\mathcal{A}_{s'}\subset\mathcal{A}_s$.

The next lemma (see \cite{Berti1}) shows the algebra property of $\mathcal{A}_s$ and interpolation inequality.
\begin{lemma}
There holds
\begin{eqnarray}\label{E2-4}
&&|AB|_s\leq c(s)|A|_s|B|_s,~~\forall A,B\in\mathcal{A}_s,~~~s>s_0>\frac{3}{2},\\
\label{E2-5}
&&|AB|_s\leq c(s)(|A|_s|B|_{s_0}+|A|_{s_0}|B|_s),~~s\geq s_0,\\
\label{E2-6}
&&\|Au\|_s\leq c(s)(|A|_s\|u\|_{s_0}+|A|_{s_0}\|u\|_s),~~\forall u\in\textbf{H}_s,~~s\geq s_0.
\end{eqnarray}
\end{lemma}
By Lemma 3, we can get, $\forall m\in\textbf{N}$,
\begin{eqnarray}\label{E2-7}
&&|A^m|_s\leq c(s)^{m-1}|A|_s^{m},\\
\label{E2-8}
&&|A^m|_s\leq m(c(s)|A|_{s_0})^{m-1}|A|_s.
\end{eqnarray}
The next two lemmas can be obtained by a small modification of the proof of Lemma 3.9 in \cite{Berti2} and Proposition 2.19 in \cite{Berti1}, so we omit it.
\begin{lemma}
Let $A\in\mathcal{A}_s$, $\Omega_1,\Omega_2\subset\Omega_{N_p}$, and $\Omega_1\cap\Omega_2=\emptyset$. Then
\begin{eqnarray*}
\|A_{\Omega_2}^{\Omega_1}\|_0\leq c(s)|A|_sd^{-1}(\Omega_1,\Omega_2)^{2s-3}.
\end{eqnarray*}
\end{lemma}
Since $\textbf{H}_s$ is an algebra, for each $b\in\textbf{H}_s$ defines the multiplication operator
\begin{eqnarray}\label{E2-3}
w(t,x)\mapsto b(t,x)w(t,x),~~\forall w\in\textbf{H}_s,
\end{eqnarray}
which is represented by $(B_b^{a})_{a,b\in\Omega_{N_p}}$ with $B_b^{a}:=\Psi_{\mathcal{N}_b}b(t,x)|_{\mathcal{N}_{a}}\in\mathcal{J}(\mathcal{N}_{a},\mathcal{N}_b)$.

\begin{lemma}
For real functions $b(t,x)\in\textbf{H}_{s+s'}$, the matrix $(B_b^{a})_{a,b\in\Omega_N}$ representing the multiplication operator (\ref{E2-3}) is self-adjoint, it belongs to the algebra of exponent localized matrices $\mathcal{A}_s$, and we have
\begin{eqnarray*}
|B|_s\leq K(s)\|b\|_{s+s'},
\end{eqnarray*}
where $K(s)$ is a constant depending on $s$.
\end{lemma}
We define
\begin{eqnarray}\label{E4-1}
h\mapsto\mathcal{J}^{(N_p)}[h]:=\Psi^{(N_p)}(\mathcal{J}_{\omega}h+2\omega^2\epsilon b(t,x)h),~~\forall h\in\textbf{H}^{(N_p)}.
%\label{E4-2}
%&&h^2\mapsto\textbf{L}^{(N)}[h^1]:=\Pi^{(N)}(L_bh^2-\varepsilon d(t,x)h^2),~~\forall h^2\in\textbf{H}^{(N)}.
\end{eqnarray}
We write (\ref{E4-1}) by the block matrix
\begin{eqnarray}\label{E4-3}
\mathcal{J}_{\omega}^{(N_p)}=D+2\omega^2\epsilon T,~~D:=diag_{a=(l,j)\leq\Omega_{N_p}}(d_{(l,j)}),
%\label{E4-4}
%&&L^{(N)}_b=D^2+\varepsilon T^2,~~D^2:=diag_{j\in J^+_N}(D_j^2I_j),
\end{eqnarray}
where $(l,j)\in\textbf{Z}^2$, $\Omega_{N_p}:=\left\{a:=(l,j)\in\textbf{Z}^2||(l,j)|\leq N_p\right\}$,
\begin{eqnarray}\label{E4-5}
d_{(l,j)}:=n\omega^2l^2-j^2,
\end{eqnarray}
and for $\mathcal{N}_a,\mathcal{N}_b\subset\Omega_{N_p}$,
\begin{eqnarray}\label{E4-7}
T:=(T^{a}_b)_{a,b\in\Omega_{N_p}},~~T^{a}_b:=\Psi_{\mathcal{N}_b}\mathcal{J}_{\omega}^{(N_p)}|_{\mathcal{N}_{a}}\in\mathcal{L}(\mathcal{N}_{a},\mathcal{N}_{b}).
\end{eqnarray}
Here $T$ is represented by the self-adjoint Toepliz matrix $(T_{a-b})_{a,b\in\Omega_{N_p}}$, the $T_a$ being the fourier coefficients of the function $b(t,x)$.

In what follows, we prove the estimate (\ref{E3-4}).
For each $N_p$, we denote the restrictions of $S$, $R$, $\Omega_{\alpha}$ to $\Omega_{N_p}$ with the same symbols.
The following result shows the separation of singular sites, and the proof can be completed by following essentially the scheme of \cite{Berti2,Berti1,Bourgain1}, so we omit it.
\begin{lemma}
Assume that $\omega$ satisfies (\ref{E1-4R}) and (\ref{E10-1}). There exists $\varsigma_0(\gamma)$ such that
for $\varsigma\in(0,\varsigma_0(\gamma)]$ and a partition of the singular sites $S$ which can be partitioned in pairwise disjoint clusters $\Omega_{\alpha}$ as
\begin{eqnarray}\label{E4-10}
S=\bigcup_{\alpha\in\textbf{A}}\Omega_{\alpha}
\end{eqnarray}
satisfying

$\bullet$ (dyadic) $\forall\alpha\in\textbf{A}\subset\Omega_N$, $M_{\alpha}\leq2m_{\alpha}$, where $M_{\alpha}:=\max_{a\in\Omega_{\alpha}}|a|$, $m_{\alpha}:=\max_{a\in\Omega_{\alpha}}|a|$.

$\bullet$ (separation) $\exists \lambda, c>0$ such that $d(\Omega_{\alpha},\Omega_{\beta})\geq c(M_{\alpha}+M_{\beta})^{\lambda}$, $\forall\alpha\neq\beta$, where $d(\Omega_{\alpha},\Omega_{\beta}):=\max_{a\in\Omega_{\alpha},b\in\Omega_{\beta}}|a-b|$ and $\lambda\in(0,1)$ .
\end{lemma}
Using Lemma 5 and Lemma 1, we have the following.
\begin{lemma}
For a real $b(t,x)\in\textbf{H}_{s'}$ with $s'>0$, the matrix $T=(T_b^{a})_{a,b\in\Omega_{N_p}}$ defined in (\ref{E4-7}) is self-adjoint and belongs to the algebra of polynomially localized matrices $\mathcal{A}_{s}$ with
\begin{eqnarray*}
|T|_s\leq  K(s)N_p^{s'}\|b\|_{s}\leq C_{\epsilon}K(s)N_p^{4+s'}\|w\|_{s},~for~s'\geq s>0,
\end{eqnarray*}
where $K(s)$ is a constant depending on $s$ and $C_{\epsilon}$ is a constant depending on $\epsilon$.
\end{lemma}
Since the decomposition
\begin{eqnarray*}
\textbf{H}^{(N_p)}:=\textbf{H}_R\oplus\textbf{H}_S,
\end{eqnarray*}
with
\begin{eqnarray*}
\textbf{H}_R:=\bigoplus_{\alpha\in R\cap\Omega_{N_p}}\mathcal{N}_{a},~~\textbf{H}_S:=\bigoplus_{\alpha\in S\cap\Omega_{N_p}}\mathcal{N}_{a},
\end{eqnarray*}
we can represent the operator $\mathcal{J}^{(N_p)}_{\omega}$ as the self-adjoint block matrix
\begin{eqnarray*}
\mathcal{J}^{(N_p)}_{\omega}=\left(
\begin{array}{ccc}
J_R& J_R^S\\
J^R_S&J_S
\end{array}
\right),
\end{eqnarray*}
where $J_R^S=(J_S^R)^{\dag}$, $J_R=J_R^{\dag}$, $J_S=J_S^{\dag}$.

Thus the invertibility of $\mathcal{J}^{(N_p)}_{\omega}$ can be expressed via the ''resolvent-type'' identity
\begin{eqnarray}\label{E4-22}
(\mathcal{J}^{(N_p)}_{\omega})^{-1}=\left(
\begin{array}{ccc}
I&-J_R^{-1}J_R^S\\
0&I
\end{array}
\right)
\left(
\begin{array}{ccc}
J_R^{-1}&0\\
0&\mathcal{J}^{-1}
\end{array}
\right)
\left(
\begin{array}{ccc}
I&0\\
-J_S^RJ_R^{-1}&I
\end{array}
\right),
\end{eqnarray}
where the ''quasi-singular'' matrix
\begin{eqnarray*}
\mathcal{J}:=J_S-J_S^RJ_R^{-1}J_R^S\in\mathcal{A}_s(S),
\end{eqnarray*}
where $\mathcal{A}_s(S)$ denotes $\mathcal{A}_s$ restricting on $S$. The reason of $\mathcal{J}\in\mathcal{A}_s(S)$ is that $\mathcal{J}$ is the restriction to $S$ of the polynomially localized matrix
\begin{eqnarray*}
I_S(J-I_SJI_R\tilde{J}^{-1}I_RJI_S)I_S\in\mathcal{A}_s,
\end{eqnarray*}
where
\begin{eqnarray*}
\tilde{J}^{-1}=\left(
\begin{array}{ccc}
I&0\\
0&J_R
\end{array}
\right).
\end{eqnarray*}
\begin{lemma}
Assume that $\omega$ satisfies (\ref{E1-4R}) and (\ref{E10-1}). For $s_0<s_1<s_2$, $|J_R^{-1}|_{s_0}\leq2\varsigma^{-1}$, the operator $J_R$ satisfies
\begin{eqnarray}\label{E4-11}
&&|\tilde{J}_R^{-1}|_{s_1}\leq c(s_1)(1+\varepsilon\varsigma^{-1}|T|_{s_1}),\\
\label{E4-12}
&&\|J_R^{-1}h\|_{s_1}\leq c(\gamma,\tau,s_2)(s_2-s_1)^{-\tau}(1+\varepsilon\varsigma^{-1}|T|_{s_2})\|h\|_{s_2},
\end{eqnarray}
where $\tilde{J}^{-1}=J^{-1}_RD_R$, $c(\gamma,\tau,s_2)$ is a constant depending on $\gamma,\tau,s_2$.
\end{lemma}
\begin{proof}
It follows from (\ref{E4-3}) and (\ref{E4-9}) that $D_R$ is a diagonal matrix and satisfies $|D_R^{-1}|_s\leq\varsigma^{-1}$. By (\ref{E2-4}), we have that the Neumann series
\begin{eqnarray}\label{E4-13}
\tilde{J}_R^{-1}=J_R^{-1}D_R=\sum_{m\geq0}(-\varepsilon)^m(D_R^{-1}T_R)^m
\end{eqnarray}
is totally convergent in $|\cdot|_{s_1}$ with $|J_R^{-1}|_{s_0}\leq2\varsigma^{-1}$, by taking $\varepsilon\varsigma^{-1}|T|_{s_0}\leq c(s_0)$ small enough.

Using (\ref{E2-4}) and (\ref{E2-8}), we have that $\forall m\in\textbf{N}$,
\begin{eqnarray*}
\varepsilon^m|(D_R^{-1}T_R)^m|_{s_1}&\leq&\varepsilon^mc(s)|(D_R^{-1}T_R)^m|_{s_1}\\
&\leq&c(s)\varepsilon^mm(c(s)|D_R^{-1}T_R|_{s_0})^{m-1}|D_R^{-1}T_R|_{s_1}\\
&\leq&c'(s)\varepsilon m\varsigma^{-1}(\varepsilon c(s_1)\varsigma^{-1}|T|_{s_0})^{m-1}|T|_{s_1},
\end{eqnarray*}
which together with (\ref{E4-13}) implies that for $\epsilon\varsigma^{-1}|T|_{s_0}< c(s_0)$ small enough, (\ref{E4-11}) holds.

By nonresonance condition (\ref{E1-4R}) and $\sup_{x>0}(x^ye^{-x})=(ye^{-1})^y$, $\forall y\geq0$, we derive
\begin{eqnarray}\label{E4-14}
e^{-2(|l|+|j|)(s_2-s_1)}|n\omega^2l^2-j^2|^{-2}&\leq&\gamma^{-1}|l|^{\tau}e^{-2(|l|+|j|)(s_2-s_1)}\nonumber\\
&\leq&c(\gamma,\tau)(s_2-s_1)^{-2\tau}.
\end{eqnarray}
Then by (\ref{E4-14}), for any $w\in\textbf{H}_R$,
\begin{eqnarray*}
\|J_R^{-1}h\|_{s_1}^2&=&\sum_{(l,j)\in R\cap\Omega_{N_p}}e^{2(|l|+|j|)s_1}\|J_R^{-1}h_j\|_{\textbf{L}^2}^2\\
&\leq&\sum_{(l,j)\in R\cap\Omega_{N_p}}e^{2(|l|+|j|)s_1}|n\omega^2l^2-j^2|^{-2}\|\tilde{J}_R^{-1}h_j\|_{\textbf{L}^2}^2\\
&\leq&\sum_{(l,j)\in R\cap\Omega_{N_p}}e^{-2(|l|+|j|)(s_2-s_1)}|n\omega^2l^2-j^2|^{-2}|e^{2(|l|+|j|)s_2}\|\tilde{J}_R^{-1}h_j\|_{\textbf{L}^2}^2\\
&\leq&c(\gamma,\tau)(s_2-s_1)^{-2\tau}\|\tilde{J}_R^{-1}h\|_{s_2}^2.
\end{eqnarray*}
Thus using interpolation (\ref{E2-6}) and (\ref{E4-11}), we derive that for $s_1<s<s_2$,
\begin{eqnarray*}
\|J_R^{-1}\|_{s_1}
&\leq&c(\gamma,\tau)(s_2-s_1)^{-\tau}\|\tilde{J}_R^{-1}h\|_{s_2}\\
&\leq&c(r,\tau,s_2)(s_2-s_1)^{\tau}(|\tilde{J}_R^{-1}|_{s_2}\|h\|_{s}+|\tilde{J}_R^{-1}|_{s}\|h\|_{s_2})\\
&\leq&c(r,\tau,s_2)(s_2-s_1)^{\tau}(1+\varepsilon\varsigma^{-1}|T|_{s_2})\|h\|_{s_2}.
\end{eqnarray*}
This completes the proof.
\end{proof}

Next we analyse the quasi-singular matrix $\mathcal{J}$. By (\ref{E4-10}), the singular sites restricted to $\Omega_{N_p}$ are
\begin{eqnarray*}
S=\bigcup_{\alpha\in l_{N_p}}\Omega_{\alpha},~~\mbox{where}~l_{N_p}:=\{\alpha\in\textbf{N}|m_{\alpha}\leq {N_p}\},
\end{eqnarray*}
and $\Omega_{\alpha}\equiv\Omega_{\alpha}\cup\Omega_{N_p}$. Due to the decomposition $\tilde{H}_S:=\bigoplus_{\alpha\in l_{N_p}}\tilde{H}_{\alpha}$, where $\textbf{H}_{\alpha}:=\bigoplus_{a\in\Omega_{\alpha}}\mathcal{N}_a$, we represent $\mathcal{J}$ as the block matrix $\mathcal{J}=(\mathcal{J}_{\alpha}^{\beta})_{\alpha,\beta\in l_{N_p}}$, where $\mathcal{J}_{\alpha}^{\beta}:=\Psi_{\textbf{H}_{\alpha}}\mathcal{J}|_{\textbf{H}_{\beta}}$.
So we can rewrite
\begin{eqnarray*}
\mathcal{J}=\mathcal{D}+\mathcal{T},
\end{eqnarray*}
where $\mathcal{D}:=diag_{\alpha\in l_N}(\mathcal{J}_{\alpha})$, $\mathcal{J}_{\alpha}:=\mathcal{J}_{\alpha}^{\alpha}$, $\mathcal{T}:=(\mathcal{J}_{\alpha}^{\beta})_{\alpha\neq\beta}$.

We define a diagonal matrix corresponding to the matrix $\mathcal{D}$ as
$\bar{D}:=diag_{\alpha\in l_{N_p}}(\bar{J}_{\alpha})$, where $\bar{J}_{\alpha}=diag_{j\in\Omega_{\alpha}}(D_j)$.
\begin{lemma}
Assume that $\omega$ satisfies (\ref{E1-4R}) and (\ref{E10-1}). We have
\begin{eqnarray*}
\|\mathcal{D}^{-1}\bar{D}h\|_{s_1}
\leq c(\varsigma,s_1,\gamma_1)N_p^{\tau}\|h\|_{s_2},
\end{eqnarray*}
where $c(\varsigma,s_1,\gamma_1)$ is a constant which depends on $\varsigma$, $s_1$ and $\gamma_1$.
\end{lemma}
\begin{proof}
Note that $\|h_{\alpha}\|_0\leq m_{\alpha}^{-s_1}\|h_{\alpha}\|_{s_1}$ and $M_{\alpha}=2m_{\alpha}$.
So for any $h=\sum_{\alpha\in l_{N_p}}h_{\alpha}\in\textbf{H}_{\alpha}$, $h_{\alpha}\in\textbf{H}_{\alpha}$,
\begin{eqnarray}\label{E4-15}
\|\mathcal{D}^{-1}\bar{D}h\|_{s_1}^2&=&\sum_{\alpha\in l_{N_p}}\|\mathcal{J}_{\alpha}^{-1}\bar{J}_{\alpha}h_{\alpha}\|_{s_1}^2\nonumber\\
&\leq&\sum_{\alpha\in l_{N_p}}M_{\alpha}^{2s_1}\|\mathcal{J}_{\alpha}^{-1}\bar{J}_{\alpha}h_{\alpha}\|_{0}^2\nonumber\\
&\leq&c\gamma_1^{-2}\sum_{\alpha\in l_{N_p}}M_{\alpha}^{2(s_1+\tau)}\|\bar{J}_{\alpha}h_{\alpha}\|_{0}^2\nonumber\\
&\leq&c\gamma_1^{-2}\sum_{\alpha\in l_{N_p}}M_{\alpha}^{2(s_1+\tau)}m_{\alpha}^{-2s_1}\|\bar{J}_{\alpha}h_{\alpha}\|_{s_1}^2\nonumber\\
&\leq&c\gamma_1^{-2}4^{s_1}\sum_{\alpha\in l_{N_p}}M_{\alpha}^{2\tau}\|\bar{J}_{\alpha}h_{\alpha}\|_{s_1}^2\nonumber\\
&\leq&c\gamma_1^{-2}4^{s_1}N_p^{2\tau}\sum_{\alpha\in l_{N_p}}\|\bar{J}_{\alpha}h_{\alpha}\|_{s_1}^2\nonumber\\
&=&c\gamma_1^{-2}4^{s_1}N_p^{2\tau}\|\bar{D}h\|_{s_1}^2.
\end{eqnarray}
In view of interpolation (\ref{E2-6}) and (\ref{E4-9}), for $0<s_1<s_2$, it follows from (\ref{E4-15}) that
\begin{eqnarray*}
\|\mathcal{D}^{-1}\bar{D}h\|_{s_1}&\leq&c\gamma_1^{-1}2^{s_1}N_p^{\tau}\|\bar{D}h\|_{s_1}\\
&\leq&c\gamma_1^{-1}2^{s_1}N_p^{\tau}(|\bar{D}|_{s_2}\|h\|_{s_1}+|\bar{D}|_{s_1}\|h\|_{s_2})\\
&\leq&c(\varsigma)\gamma_1^{-1}2^{s_1+1}N_p^{\tau}\|h\|_{s_2}.
\end{eqnarray*}
This completes the proof.
\end{proof}

The following result is taken from \cite{Berti1}, so we omit the proof.
\begin{lemma}
$\forall s\geq0$, $\forall m\in\textbf{N}$, there hold:
\begin{eqnarray}\label{E4-17}
&&c(s_1)\|\mathcal{D}^{-1}\mathcal{T}\|_{s_0}<\frac{1}{2},~~\|\mathcal{D}^{-1}\|_s\leq c(s)\gamma_1^{-1}N_p^{\tau},\\
\label{E4-19}
&&\|(N_p^{-4}\mathcal{D}^{-1}\mathcal{T})^mh\|_s\leq(\varepsilon\gamma^{-1}N_p^{-4}K(s))^m(mN^{\kappa_0}|T|_s|T|_{s_0}^{m-1}\|h\|_{s_0}+|T|^m_{s_0}\|h\|_s).~~~~~~~
\end{eqnarray}
\end{lemma}
\begin{lemma}
Assume that $\omega$ satisfies (\ref{E1-4R}) and (\ref{E10-1}). For $\kappa_0=\tau+3$ and $0<s_0<s_1<s_2<s_3$, we have
\begin{eqnarray}
\label{E4-18}
\|\mathcal{J}^{-1}h\|_{s_1}\leq c(\varsigma,\tau,s_1,\gamma_1,\gamma)N_p^{\tau+\kappa_0}(s_3-s_2)^{-\tau}(\|h\|_{s_3}+\varepsilon|T|_{s_1}\|h\|_{s_2}).
\end{eqnarray}
\end{lemma}
\begin{proof}
The Neumann series
\begin{eqnarray}\label{E4-16R}
\mathcal{J}^{-1}=(N_p^{-4}I+N_p^{-4}\mathcal{D}^{-1}\mathcal{T})^{-1}N_p^{-4}\mathcal{D}^{-1}
=(N_p^{-4}I+\sum_{m\geq1}(-1)^m(N_p^{-4}\mathcal{D}^{-1}\mathcal{T})^m)N_p^{-4}\mathcal{D}^{-1}~~~
\end{eqnarray}
is totally convergent in operator norm $\|\cdot\|_{s_0}$ with $\|\mathcal{J}^{-1}\|_{s_0}\leq c\gamma_1^{-1}N_p^{\tau}$,
by using (\ref{E4-17}).

By (\ref{E4-19}) and (\ref{E4-16R}), we have
\begin{eqnarray}\label{E4-16}
\|\mathcal{J}^{-1}h\|_{s_1}&\leq&\|\mathcal{D}^{-1}h\|_{s_1}+N_p^{-4}
\sum_{m\geq1}\|(N_p^{-4}\mathcal{D}^{-1}\mathcal{T})^m\mathcal{D}^{-1}h\|_{s_1}\nonumber\\
&\leq&\|\mathcal{D}^{-1}h\|_{s_1}+
\|\mathcal{D}^{-1}h\|_{s_1}\sum_{m\geq1}(\varepsilon\gamma_1^{-1}N_p^{-4}K(s)|T|_{s_0})^m\nonumber\\
&&+N_p^{\kappa_0}K(s_1)\varepsilon\gamma_1^{-1}|T|_{s_1}\|\mathcal{D}^{-1}h\|_{s_0}\sum_{m\geq1}m(N_p^{-4}K(s)\varepsilon\gamma_1^{-1}|T|_{s_0})^{m-1}.~~~~~~~~~
\end{eqnarray}
In terms of $\sup_{x>0}(x^ye^{-x})=(ye^{-1})^y$, $\forall y\geq0$, for $0<s_1<s_2<s_3$, it follows from Lemma 9 that
\begin{eqnarray}\label{E4-20}
\|\mathcal{D}^{-1}h\|^2_{s_1}&=&\|\mathcal{D}^{-1}\bar{D}\bar{D}^{-1}h\|^2_{s_1}\nonumber\\
&\leq& c^2(\varsigma,s_1,\gamma_1)N_p^{2\tau}\|\bar{D}^{-1}h\|^2_{s_2}\nonumber\\
&=&c^2(\varsigma,s_1,\gamma_1)N_p^{2\tau}\sum_{(l,j)\in R\cap\Omega_{N_p}}e^{2(|l|+|j|)s_2}\|\bar{D}^{-1}h_j\|_{\textbf{L}^2}^2\nonumber\\
&\leq&c^2(\varsigma,s_1,\gamma_1)N_p^{2\tau}\sum_{(l,j)\in R\cap\Omega_{N_p}}e^{2(|l|+|j|)s_2}|n\omega^2l^2-j^2|^{-2}\|h_j\|_{\textbf{L}^2}^2\nonumber\\
&\leq&c^2(\varsigma,s_1,\gamma_1)N_p^{2\tau}\sum_{(l,j)\in R\cap\Omega_{N_p}}e^{-2(|l|+|j|)(s_3-s_2)}|l|^{-2}e^{2(|l|+|j|)s_3}\|h_j\|_{\textbf{L}^2}^2~~~~\nonumber\\
&\leq&c^2(\varsigma,\tau,s_1,\gamma_1,\gamma)N_p^{2\tau}(s_3-s_2)^{-2\tau}\|h\|_{s_3}^2.
\end{eqnarray}
Thus by (\ref{E4-16}) and (\ref{E4-20}), we derive
\begin{eqnarray}\label{E4-21}
\|\mathcal{J}^{-1}h\|_{s_1}
&\leq&\gamma_1^{-1}N_p^{\kappa_0}K'(s_1)(\|\mathcal{D}^{-1}h\|_{s_1}+\varepsilon|T|_{s_1}\|\mathcal{D}^{-1}h\|_{s_0})\nonumber\\
&\leq&c(\varsigma,\tau,s_1,\gamma_1,\gamma)N_p^{\tau+\kappa_0}(s_3-s_2)^{-\tau}(\|h\|_{s_3}+\varepsilon|T|_{s_1}\|h\|_{s_2}),~~~~~
\end{eqnarray}
where $0<s_1<s_2<s_3$ and $\varepsilon\gamma_1^{-1}\varsigma^{-1}(1+|T|_{s_0})\leq c(k)$ small enough.
\end{proof}

Now we are ready to prove Lemma 2. Let
\begin{eqnarray*}
h=h_R+h_S
\end{eqnarray*}
with $h_S\in\textbf{H}_S$ and $h_R\in\textbf{H}_R$. Then by the resolvent identity (\ref{E4-22}),
\begin{eqnarray}\label{E4-23}
\|(\mathcal{J}_{\omega}^{(N_p)})^{-1}h\|_{s_1}&\leq&\|J_R^{-1}h_R+J_R^{-1}J_S^R\mathcal{J}^{-1}(h_S+J_{R}^SJ_R^{-1}h_R)\|_{s_1}
+\|\mathcal{J}^{-1}(h_R+J_R^SJ_R^{-1}h_R)\|_{s_1}\nonumber\\
&\leq&\|J_R^{-1}h_R\|_{s_1}+\|J_R^{-1}J_S^R\mathcal{J}^{-1}h_S\|_{s_1}+\|J_R^{-1}J_S^R\mathcal{J}^{-1}J_{R}^SJ_R^{-1}h_R\|_{s_1}\nonumber\\
&&+\|\mathcal{J}^{-1}h_R\|_{s_1}+\|\mathcal{J}^{-1}J_R^SJ_R^{-1}h_R\|_{s_1}.
\end{eqnarray}
Next we estimate the right hand side of (\ref{E4-23}) one by one. Using (\ref{E2-6}), (\ref{E4-12}) and (\ref{E4-18}), for $0<s_1<s_2<s_3<s_4$, we have
\begin{eqnarray}\label{E4-24}
\|J_R^{-1}J_S^R\mathcal{J}^{-1}h_S\|_{s_1}&\leq&c(\gamma,\tau,s_2)(s_2-s_1)^{-\tau}(1+\varepsilon\varsigma^{-1}|T|_{s_2})\|J_S^R\mathcal{J}^{-1}h_S\|_{s_2}\nonumber\\
&\leq&c(\gamma,\tau,s_2)(s_2-s_1)^{-\tau}(1+\epsilon\varsigma^{-1}|T|_{s_2})|T|_{s_2}\|\mathcal{J}^{-1}h\|_{s_2}\nonumber\\
&\leq&c(\gamma,\gamma_1,\varsigma,\tau,s_2)(s_2-s_1)^{-\tau}(s_4-s_3)^{-\tau}N_p^{\tau+\kappa_0}\nonumber\\
&&\times(1+\epsilon\varsigma^{-1}|T|_{s_2})|T|_{s_2}(\|h\|_{s_3}+\varepsilon|T|_{s_2}\|h\|_{s_4}),
\end{eqnarray}
\begin{eqnarray}\label{E4-25}
\|\mathcal{J}^{-1}J_R^SJ_R^{-1}h_R\|_{s_1}&\leq& c(\varsigma,\tau,s_1,\gamma_1,\gamma)N_p^{\tau+\kappa_0}(s_3-s_2)^{-\tau}(\|J_R^SJ_R^{-1}h_R\|_{s_3}+\epsilon|T|_{s_1}\|J_R^SJ_R^{-1}h_R\|_{s_2})\nonumber\\
&\leq&c(\varsigma,\tau,s_1,s_2,s_3,\gamma_1,\gamma)N_p^{\tau+\kappa_0}(s_3-s_2)^{-\tau}(|T|_{s_3}\|J_R^{-1}h_R\|_{s_3}+\varepsilon|T|_{s_1}|T|_{s_2}\|J_R^{-1}h_R\|_{s_2})\nonumber\\
&\leq&c(\varsigma,\tau,s_1,s_2,s_3,\gamma_1,\gamma)N_p^{\tau+\kappa_0}(s_3-s_2)^{-\tau}(|T|_{s_3}(s_4-s_3)^{-\tau}(1+\epsilon\varsigma^{-1}|T|_{s_4})\|h\|_{s_4}\nonumber\\
&&+\epsilon|T|_{s_1}|T|_{s_2}(s_3-s_2)^{-\tau}(1+\varepsilon\varsigma^{-1}|T|_{s_3})\|h\|_{s_3})\nonumber\\
&\leq&c(\varsigma,\tau,s_1,s_2,s_3,\gamma_1,\gamma)N_p^{\tau+\kappa_0}(s_3-s_2)^{-\tau}|T|_{s_3}(1+\epsilon\varsigma^{-1}|T|_{s_4})\nonumber\\
&&\times((s_4-s_3)^{-\tau}\|h\|_{s_4}+\epsilon|T|_{s_2}(s_3-s_2)^{-\tau}\|h\|_{s_3}),
\end{eqnarray}
\begin{eqnarray}\label{E4-26}
\|J_R^{-1}J_S^R\mathcal{J}^{-1}J_{R}^SJ_R^{-1}h_R\|_{s_1}&\leq& c(\gamma,\tau,s_2)(s_2-s_1)^{-\tau}(1+\epsilon\varsigma^{-1}|T|_{s_2})\|J_S^R\mathcal{J}^{-1}J_{R}^SJ_R^{-1}h_R\|_{s_2}\nonumber\\
&\leq&c(\gamma,\tau,s_2)(s_2-s_1)^{-\tau}(1+\epsilon\varsigma^{-1}|T|_{s_2})|T|_{s_2}\|\mathcal{J}^{-1}J_{R}^SJ_R^{-1}h_R\|_{s_2}\nonumber\\
&\leq&c(\varsigma,\tau,s_1,s_2,s_3,\gamma_1,\gamma)N_p^{\tau+\kappa_0}(s_3-s_2)^{-\tau}(s_2-s_1)^{-\tau}|T|^2_{s_3}\nonumber\\
&&\times(1+\epsilon\varsigma^{-1}|T|_{s_4})^2((s_4-s_3)^{-\tau}\|h\|_{s_4}\nonumber\\
&&+\epsilon|T|_{s_2}(s_3-s_2)^{-\tau}\|h\|_{s_3}).
\end{eqnarray}
The terms $\|J_R^{-1}h_R\|_{s_1}$ and $\|\mathcal{J}^{-1}h_R\|_{s_1}$ can be controlled by using (\ref{E4-12}) and (\ref{E4-18}). Thus by (\ref{E4-23})-(\ref{E4-26}), for $0<s<\tilde{s}$, we conclude
\begin{eqnarray*}
\|(\mathcal{J}_{\omega}^{(N_p)})^{-1}h\|_{s}\leq c(\varsigma,\tau,s,\tilde{s},\gamma_1,\gamma)N_p^{\tau+\kappa_0}(1+\epsilon\varsigma^{-1}|T|_{\tilde{s}})^3(\tilde{s}-s)^{-\tau}\|h\|_{\tilde{s}},
\end{eqnarray*}
which together with Lemma 7 gives (\ref{E3-4}).

\section{Proof of Main Result}
In this section, we give the proof of Theorem 1.
The proof is based on constructing a suitable Nash-Moser iteration scheme.
The dependence upon the parameter $\omega$ of the solution of (\ref{E1-7}), as is well known, is more delicated
since it involves in the small divisors of $\omega_j$: it is, however, standard to check
that this dependence is $\textbf{C}^1$ on a bounded set of
Diophantine numbers, for more details, see, for example, \cite{Berti2,Berti1}.

We construct the first step approximation.
\begin{lemma}
Assume that $\omega$ satisfies (\ref{E1-4R}) and (\ref{E10-1}). Then system
(\ref{E3-6}) has the first step approximation
$w^1\in \textbf{H}_s^{(N_1)}$
\begin{eqnarray}\label{E3-8}
w^1=2\omega^2\epsilon(I-\Psi^{(N_0)})\Psi^{(N_1)}\left(\sum_{k=1}^n(w^0_{ky}+\frac{\epsilon}{2}|w^0_y|^2)w^0_{tt}-(\sum_{k=1}^nw^0_{kt}+\epsilon\langle w^0_t, w^0_y\rangle)w^0_{ty}+\frac{\epsilon}{2}|w^0_t|^2w^0_{yy}\right),~~~
\end{eqnarray}
and the error term
\begin{eqnarray}
\label{E3-10}
E^1&=&R^0=2\omega^2\epsilon\Psi^{(N_1)}(\sum_{k=1}^nw^1_{ky}w^1_{tt}+\frac{\epsilon}{2}(2\langle w^0_y,w^1_y\rangle+|w^1_y|^2)w^1_{tt}+\epsilon|w^1_y|^2w^0_{tt}+\sum_{k=1}^nw^1_{kt}w^1_{ty}\nonumber\\
&&+\epsilon((\langle w^0_t,w^1_y\rangle+\langle w^0_y,w^1_t\rangle)w^1_{ty}+\langle w^1_t,w^1_y\rangle w^0_{ty})+2\epsilon(\langle w^0_t,w^1_t\rangle+|w^1_t|^2)w^1_{yy}).
%\label{E3-11}
%E_1^{(2)}&=&R_0^{(2)}=-\varepsilon (g(x,u_0+u_1,v_0+v_1)-g(u_0,v_0)\nonumber\\
%&&-(D_ug(x,u_0,v_0)u_1+D_vg(x,u_0,v_0)v_1)).
\end{eqnarray}
\end{lemma}
\begin{proof}
Assume that we have chosen suitable the $0th$ step approximation
solution $w^0\neq C$. Then the target is to get the $1th$ step
approximation solution.

Denote
\begin{eqnarray}\label{E3-12R}
E^0=\mathcal{J}_{\omega}w^0+2\omega^2\epsilon\Psi^{(N_1)}f(w^0_t,w_y^0,w^0_{tt},w_{yy}^0).
\end{eqnarray}
By (\ref{E3-6}), we have
\begin{eqnarray}\label{E3-12}
\mathcal{J}(w^0+w^1)&=&\mathcal{J}_{\omega}(w^0+w^1)+2\omega^2\epsilon\Psi^{(N_1)}f(w^0_t+w^1_t,w_y^0+w_y^1,w^0_{tt}+w^1_{tt},w^0_{yy}+w^1_{yy})\nonumber\\
&=&\mathcal{J}_{\omega}w^0+2\omega^2\epsilon\Psi^{(N_1)}f(w^0_t,w_y^0,w^0_{tt},w_{yy}^0)+\mathcal{J}_{\omega}w^1
+2\omega^2\epsilon\Psi^{(N_1)}D_wf(w^0_t,w_y^0,w^0_{tt},w_{yy}^0)w^1\nonumber\\
&&+2\omega^2\epsilon\Psi^{(N_1)}(f(w^0_t+w^1_t,w_y^0+w_y^1,w^0_{tt}+w^1_{tt},w^0_{yy}+w^1_{yy})-f(w^0_t,w_y^0,w^0_{tt},w_{yy}^0)\\
&&-D_wf(w^0_t,w_y^0,w^0_{tt},w_{yy}^0)w^1)\nonumber\\
&=&E^0+\mathcal{J}^{(N_1)}_{\omega}w^1+R^0,
\end{eqnarray}
where
\begin{eqnarray*}
\mathcal{J}^{(N_1)}_{\omega}w^1&=&\mathcal{J}_{\omega}w^1+2\omega^2\epsilon\Psi^{(N_1)}D_wf(w^0_t,w_y^0,w^0_{tt},w_{yy}^0)w^1\\
&=&\mathcal{J}_{\omega}w^1+2\omega^2\epsilon\Psi^{(N_1)}((\sum_{k=1}^nw^0_{ky}+\frac{\epsilon}{2}|w^0_y|^2)w^1_{tt}-(\sum_{k=1}^nw^0_{ky}+\epsilon\langle w^0_t,w^0_y\rangle)w^1_{ty}+\frac{\epsilon}{2}|w^0_t|^2w^1_{yy}+\sum_{k=1}^nw^1_{ky}w^0_{tt}\nonumber\\
&&+2\epsilon\langle w^0_y,w^1_y\rangle w^0_{tt}-\sum_{k=1}^nw^1_{kt}w^0_{ty}-\epsilon(\langle w^0_t,w^1_y\rangle+\langle w^1_t,w^0_y\rangle)w^0_{ty}+\epsilon\langle w^0_t,w^1_t\rangle w^0_{yy}),
\end{eqnarray*}
\begin{eqnarray*}
R^0&=&2\omega^2\epsilon\Psi^{(N_1)}(f(w^0_t+w^1_t,w_y^0+w_y^1,w^0_{tt}+w^1_{tt},w^0_{yy}+w^1_{yy})-f(w^0_t,w_y^0,w^0_{tt},w_{yy}^0)-D_wf(w^0_t,w_y^0,w^0_{tt},w_{yy}^0)w^1)\\
&=&2\omega^2\epsilon\Psi^{(N_1)}((\sum_{k=1}^nw^1_{ky})w^1_{tt}+\frac{\epsilon}{2}(2\langle w^0_y,w^1_y\rangle+|w^1_y|^2)w^1_{tt}+\epsilon|w^1_y|^2w^0_{tt}\nonumber\\
&&+\epsilon((\langle w^0_t,w^1_y\rangle+\langle w^0_y,w^1_t\rangle)w^1_{ty}+\langle w^1_t,w^1_y\rangle w^0_{ty})+(\sum_{k=1}^nw^1_{kt})w^1_{ty}+2\epsilon(\langle w^0_t,w^1_t\rangle+|w^1_t|^2)w^1_{yy}).
\end{eqnarray*}
Then taking
\begin{eqnarray}\label{E6-3}
w^1=-(\mathcal{J}^{(N_1)}_{\omega})^{-1}E^0\in\textbf{H}_s^{(N_1)},
\end{eqnarray}
thus it has
\begin{eqnarray*}
E^0+\mathcal{J}^{(N_1)}_{\omega}w^1=0.
\end{eqnarray*}
By (\ref{E3-12}), we deduce
\begin{eqnarray*}
E^1&:=&R^0=\mathcal{J}(w^0+w^1)\nonumber\\
&=&2\omega^2\epsilon\Psi^{(N_1)}(f(w^0_t+w^1_t,w_y^0+w_y^1,w^0_{tt}+w^1_{tt},w^0_{yy}+w^1_{yy})-f(w^0_t,w_y^0,w^0_{tt},w_{yy}^0)-D_wf(w^0_t,w_y^0,w^0_{tt},w_{yy}^0)w^1)\nonumber\\
&=&2\omega^2\epsilon\Psi^{(N_1)}((\sum_{k=1}^nw^1_{ky})w^1_{tt}+\frac{\epsilon}{2}(2\langle w^0_y,w^1_y\rangle+|w^1_y|^2)w^1_{tt}+\epsilon|w^1_y|^2w^0_{tt}+\sum_{k=1}^nw^1_{kt}w^1_{ty}\nonumber\\
&&+\epsilon((\langle w^0_t,w^1_y\rangle+\langle w^0_y,w^1_t\rangle)w^1_{ty}+\langle w^1_t,w^1_y\rangle w^0_{ty})+2\epsilon(\langle w^0_t,w^1_t\rangle+|w^1_t|^2)w^1_{yy}).
\end{eqnarray*}
In fact, by (\ref{E3-6}) and (\ref{E3-12R}), we can obtain
\begin{eqnarray}\label{E3-12R1}
E^0=2\omega^2\epsilon(I-\Psi^{(N_0)})\Psi^{(N_1)}f(w^0_t,w_y^0,w^0_{tt},w_{yy}^0).
\end{eqnarray}
The proof is thus complete.
\end{proof}

In order to prove the convergence of the Newton algorithm,  the
following estimate is needed. Firstly, for convenience, we define
\begin{eqnarray}\label{E3-12R2}
\tilde{E}^0:=2\omega^2\epsilon\Psi^{(N_1)}f(w^0_t,w_y^0,w^0_{tt},w_{yy}^0)\neq0.
\end{eqnarray}

The following result tells us that for choosing suitable initial step $w^0$ and sufficient small $\epsilon$, we can control the term $w^1$ and $E^1$ under some norm.
\begin{lemma}
Assume that $\omega$ satisfies (\ref{E1-4R}) and (\ref{E10-1}). Then for any $0<\alpha<\sigma$, the
following estimates hold:
\begin{eqnarray*}
\|w^1\|_{\sigma-\alpha}\leq2\omega^2\epsilon C(\alpha)C_{\epsilon}N_1^{\tau+\kappa_0+4}\|w^0\|_{\sigma}^3(1+\epsilon cN_1^{4+\sigma}\|w^0\|_{\sigma})^3,
\end{eqnarray*}
\begin{eqnarray*}
\label{E3-13}
\|E^1\|_{\sigma-\alpha}\lesssim C_{\epsilon,\omega,\alpha}
\epsilon^3N_1^{12+2(\tau+\kappa_0)}\|w^0\|^6_{\sigma},
\end{eqnarray*}
where $C_{\epsilon,\omega,\alpha}$ is a positive constant depending on $\epsilon$, $\omega$ and $\alpha$.
\end{lemma}
\begin{proof}
Denote
\begin{eqnarray}\label{E2-12'}
C(\alpha)&=&c(\varsigma,\tau,\sigma,\gamma_1,\gamma)\alpha^{-\tau}.
\end{eqnarray}
From the definition of $w^1$ in
(\ref{E3-8}), by Lemma 2, (\ref{E3-1R2}),  (\ref{E3-3R4}) and (\ref{E3-12R2}), we derive
\begin{eqnarray}\label{E3-14}
\|w^1\|_{\sigma-\alpha}&=&\|-(\mathcal{J}^{( N_1)}_{\omega})^{-1}E^0\|_{\sigma-\alpha}\nonumber\\
&\leq&C(\alpha)N_1^{\tau+\kappa_0}(1+\epsilon cN_1^{4+\sigma}\|w^0\|_{\sigma})^3\|E^0\|_{\sigma}\nonumber\\
&\leq&2\omega^2\epsilon C(\alpha)(1+\epsilon cN_1^{4+\sigma}\|w^0\|_{\sigma})^3\|\Psi^{(N_1)}f(w^0_t,w_y^0,w^0_{tt},w_{yy}^0)\|_{\sigma+\tau+\kappa_0}\nonumber\\
&\leq&2\omega^2\epsilon C(\alpha)C_{\epsilon}N_1^{\tau+\kappa_0+4}\|w^0\|_{\sigma}^3(1+\epsilon cN_1^{4+\sigma}\|w^0\|_{\sigma})^3.
\end{eqnarray}
By (\ref{E3-1R2}), (\ref{E3-3R3}), (\ref{E3-14}) and the
definition of $E^1$, we have
\begin{eqnarray*}
\|E^1\|_{\sigma-\alpha}
&\leq&2\omega^2\epsilon C_{\epsilon}N_1^{4}(C_3+\|w^0\|_{\sigma-\alpha})(\|w^1\|_{\sigma-\alpha}^2+\|w^1\|_{\sigma-\alpha}^3)\nonumber\\
&=&2\omega^2\epsilon C_{\epsilon}N_1^{4}(C_3+\|w^0\|_{\sigma-\alpha})\|w^1\|_{\sigma-\alpha}^2(1+\|w^1\|_{\sigma-\alpha})\nonumber\\
&\lesssim&C_{\epsilon,\omega,\alpha}
\epsilon^3N_1^{12+2(\tau+\kappa_0)}\|w^0\|^6_{\sigma}.
\end{eqnarray*}
This completes the proof.
\end{proof}
For $p\in\textbf{N}$ and $0<\sigma_{0}<\bar{\sigma}<\sigma$, set
\begin{eqnarray}\label{E3-16}
&&\sigma_p:=\bar{\sigma}+\frac{\sigma-\bar{\sigma}}{2^p},\\
\label{E3-17}
&&\alpha_{p+1}:=\sigma_p-\sigma_{p+1}=\frac{\sigma-\bar{\sigma}}{2^{p+1}}.
\end{eqnarray}
By (\ref{E3-16}) and (\ref{E3-17}), it follows that
\begin{eqnarray*}
\sigma_0>\sigma_1>\ldots>\sigma_p>\sigma_{p+1}>\ldots,~for~p\in\textbf{N}.
\end{eqnarray*}
Define
\begin{eqnarray*}
&&\mathcal{P}_1(w^0):=w^0+w^1,~~for~w^0\in \textbf{H}_{\sigma_0}^{(N_0)},\\
&&E^{p}=\mathcal{J}(\sum_{k=0}^pw^k)=\mathcal{J}(\mathcal{P}_1\underbrace{\circ\cdots\circ}_p\mathcal{P}_1(w^0)).
\end{eqnarray*}
In fact, to obtain the $p$ th approximation solution
$w^p\in\textbf{H}_{\sigma_p}^{(N_p)}$ of system
(\ref{E3-6}), as done in lemma 12, we need to solve following equations
\begin{eqnarray*}
\mathcal{J}(\sum_{k=0}^pw^k)&=&\mathcal{J}_{\omega}(\sum_{k=0}^{p-1}
w^k)+2\omega^2\epsilon\Psi^{(N_p)}f(\sum_{k=0}^{p-1}w^k_t,\sum_{k=0}^{p-1}w^k_y,\sum_{k=0}^{p-1}w^k_{tt},\sum_{k=0}^{p-1}w^k_{yy})+\mathcal{J}_{\omega}w^p\nonumber\\
&&+2\omega^2\epsilon\Psi^{(N_p)}D_wf(\sum_{k=0}^{p-1}w^k_t,\sum_{k=0}^{p-1}w^k_y,\sum_{k=0}^{p-1}w^k_{tt},\sum_{k=0}^{p-1}w^k_{yy})w^p
+2\omega^2\epsilon\Psi^{(N_p)}(f(\sum_{k=0}^{p}w^k_t,\sum_{k=0}^{p}w^k_y,\sum_{k=0}^{p}w^k_{tt},\sum_{k=0}^{p}w^k_{yy})\nonumber\\
&&-f(\sum_{k=0}^{p-1}w^k_t,\sum_{k=0}^{p-1}w^k_y,\sum_{k=0}^{p-1}w^k_{tt},\sum_{k=0}^{p-1}w^k_{yy})-D_wf(\sum_{k=0}^{p-1}w^k_t,\sum_{k=0}^{p-1}w^k_y,\sum_{k=0}^{p-1}w^k_{tt},\sum_{k=0}^{p-1}w^k_{yy})).
\end{eqnarray*}
Then the $p$ th step approximation $w^p\in
\textbf{H}_{\sigma_p}^{(N_p)}$ :
\begin{eqnarray}\label{E2-32}
w^p=-(\mathcal{J}_{\omega}^{(N_p)})^{-1}E^{p-1},
\end{eqnarray}
where
\begin{eqnarray*}
E^p&=&\mathcal{J}_{\omega}(\sum_{k=0}^{p-1}
w^k)+2\omega^2\epsilon\Psi^{(N_p)}f(\sum_{k=0}^{p-1}w^k_t,\sum_{k=0}^{p-1}w^k_y,\sum_{k=0}^{p-1}w^k_{tt},\sum_{k=0}^{p-1}w^k_{yy})\nonumber\\
&=&2\omega^2\epsilon(I-\Psi^{(N_{p-1})})\Psi^{(N_p)}f(\sum_{k=0}^{p-1}w^k_t,\sum_{k=0}^{p-1}w^k_y,\sum_{k=0}^{p-1}w^k_{tt},\sum_{k=0}^{p-1}w^k_{yy}).
\end{eqnarray*}
As done in Lemma 12, it is easy to get that
\begin{eqnarray}\label{E2-32r}
E^p:=R^{p-1}&=&2\omega^2\epsilon\Psi^{(N_p)}(\sum_{k=1}^nw^p_{ky}w^p_{tt}+\frac{\epsilon}{2}(2\langle \sum_{k=0}^{p-1}w^{k}_y,w^{p}_y\rangle+|w^{p}_y|^2)w^{p}_{tt}+\epsilon|w^{p}_y|^2\sum_{k=0}^{p-1}w^{k}_{tt}\nonumber\\
&&+\sum_{k=1}^nw^{p}_{kt}w^{p}_{ty}+\epsilon((\langle \sum_{k=0}^{p-1}w^{k}_t,w^{p}_y\rangle+\langle \sum_{k=0}^{p-1}w^{k}_y,w^{p}_t\rangle)w^{p}_{ty}+\langle w^{p}_t,w^{p}_y\rangle\sum_{k=0}^{p-1}w^{k}_{ty})\nonumber\\
&&+2\epsilon(\langle \sum_{k=0}^{p-1}w^{k}_t,w^p_t\rangle+|w^p_t|^2)w^p_{yy}),~~
\end{eqnarray}
\begin{eqnarray}
\label{E2-32rr}
\tilde{E}^p=2\omega^2\epsilon\Psi^{(N_p)}f(\sum_{k=0}^{p-1}w^k_t,\sum_{k=0}^{p-1}w^k_y,\sum_{k=0}^{p-1}w^k_{tt},\sum_{k=0}^{p-1}w^k_{yy}).
\end{eqnarray}
Hence we only need to estimate $R^{p-1}$
to prove the convergence of algorithm.
The following result tells us that the existence of solutions for
(\ref{E3-6}). A key point is to give a
sufficient condition on the convergence of Newton algorithm.
\begin{lemma}
Assume that $\omega$ satisfies (\ref{E1-4R}) and (\ref{E10-1}). Then, for sufficiently small $\epsilon$, equations
(\ref{E3-6}) has a solution
\begin{eqnarray*}
w^{\infty}=\sum_{k=0}^{\infty}w^k\in
\textbf{H}_{\bar{\sigma}}.
\end{eqnarray*}
\end{lemma}
\begin{proof}
This proof is based on the induction.
Firstly, we claim that
\begin{eqnarray}\label{E6-1}
\|w^p\|_{\sigma_p}<1~~and~~\|E^p\|_{\sigma_p}\leq d^{4^p},~\forall p\in\textbf{N},
\end{eqnarray}
where $d$ denotes a constant which takes value in $(0,1)$.

We prove (\ref{E6-1}) by induction. For $p=0$, we choose an initial approximation solution $w^0$ such that
\begin{eqnarray*}
&&w^0\neq C,~~0<\|w^0\|_{\sigma_0}<R,\\
&&\|E^0\|_{\sigma_0}\leq d^{4^p},~~for~some~0<d<1.
\end{eqnarray*}

For $p=1$, by Lemma 13, we can also choose a sufficient small $\epsilon$ such that (\ref{E6-1}) holds.

Assume that (\ref{E6-1}) holds for $1\leq k\leq p-1$, i.e.
\begin{eqnarray}\label{E6-2}
\|w^k\|_{\sigma_k}<1~~and~~\|E^k\|_{\sigma_k}\leq d^{4^k}.
\end{eqnarray}
We prove that $k=p$ holds. For any fixed number $p\in\textbf{N}$, we denote
\begin{eqnarray}\label{E11-1}
\max_{0\leq k\leq p}\|w^k\|_{\sigma_k}:=\|w^{k^*}\|_{\sigma_{k^*}},~~\max_{0\leq k\leq p-1}\|w^k\|_{\sigma_k}:=\|w^{k^{**}}\|_{\sigma_{k^{**}}},~0\leq k^*\leq p.
\end{eqnarray}
For convenience, we assume that $k^*=p$ and $k^{**}=p-1$. Using the same idea of the proof of $k^*=p$ and $k^{**}=p-1$, we can verify other cases. In fact, the proof process of the cases of $k^*\neq p$ or $k^{**}\neq p-1$ is more simple than the case of $k^*=p$ and $k^{**}=p-1$.

This proof is divided into the following situations. If
\begin{eqnarray*}
\epsilon cN_p^{4+\sigma_p}\sum_{k=0}^{p-1}\|w^{k}\|_{\sigma_{k}}<1,
\end{eqnarray*}
we have
\begin{eqnarray*}
&&Case~1:~2C(\alpha_p)N_p^{\tau+\kappa_0}\|E^{p-1}\|_{\sigma_{p-1}}>1,~\sum_{k=0}^{p-1}\|w^{k}\|_{\sigma_{p}}>C_{3,R'},\\
&&Case~2:~2C(\alpha_p)N_p^{\tau+\kappa_0}\|E^{p-1}\|_{\sigma_{p-1}}>1,~\sum_{k=0}^{p-1}\|w^{k}\|_{\sigma_{p}}\leq C_{3,R'},\\
&&Case~3:~2C(\alpha_p)N_p^{\tau+\kappa_0}\|E^{p-1}\|_{\sigma_{p-1}}\leq1,~\sum_{k=0}^{p-1}\|w^{k}\|_{\sigma_{p}}>C_{3,R'},\\
&&Case~4:~2C(\alpha_p)N_p^{\tau+\kappa_0}\|E^{p-1}\|_{\sigma_{p-1}}\leq1,~\sum_{k=0}^{p-1}\|w^{k}\|_{\sigma_{p}}\leq C_{3,R'}.
\end{eqnarray*}
We only prove the case $1$, the rest cases are similar. By (\ref{E3-4}) in Lemma 2, (\ref{E2-32}) and (\ref{E2-32rr}), we derive
\begin{eqnarray}\label{E3-18}
\|w^p\|_{\sigma_p}&=&\|-(\mathcal{J}_{\omega}^{(N_p)})^{-1}E^{p-1}\|_{\sigma_p}\nonumber\\
&\leq&C(\alpha_p)N_{p}^{\tau+\kappa_0}\|E^{p-1}\|_{\sigma_{p-1}}\left(1+\epsilon cN_p^{4+\sigma_p}\sum_{k=0}^{p-1}\|w^{k}\|_{\sigma_{k}}\right)^3\nonumber\\
&\leq&C(\alpha_p)N_{p}^{\tau+\kappa_0}\|E^{p-1}\|_{\sigma_{p-1}}.
\end{eqnarray}
By (\ref{E3-3R3}) in Lemma 1, (\ref{E3-17}), (\ref{E2-32r})--(\ref{E3-18}), we have
\begin{eqnarray}\label{E3-20}
\|E^p\|_{\sigma_p}&=&2\omega^2\epsilon\|\Psi^{(N_p)}(f(\sum_{k=0}^{p}w^k_t,\sum_{k=0}^{p}w^k_y,\sum_{k=0}^{p}w^k_{tt},\sum_{k=0}^{p}w^k_{yy})
-f(\sum_{k=0}^{p-1}w^k_t,\sum_{k=0}^{p-1}w^k_y,\sum_{k=0}^{p-1}w^k_{tt},\sum_{k=0}^{p-1}w^k_{yy})\nonumber\\
&&-D_wf(\sum_{k=0}^{p-1}w^k_t,\sum_{k=0}^{p-1}w^k_y,\sum_{k=0}^{p-1}w^k_{tt},\sum_{k=0}^{p-1}w^k_{yy})\|_{\sigma_p}\nonumber\\
&\leq&2\omega^2\epsilon C_{\epsilon}N_p^4\|w^p\|_{\sigma_p}^2(C_{3,R'}+\sum_{k=0}^{p-1}\|w^{k}\|_{\sigma_k})(1+\|w^p\|_{\sigma_p}^2)\nonumber\\
&\leq&2\omega^2\epsilon C_{\epsilon}C^2(\alpha_p)N_{p}^{2(\tau+\kappa_0+2)}\|E^{p-1}\|^2_{\sigma_{p-1}}
(C_{3,R'}+\sum_{k=0}^{p-1}\|w^{k}\|_{\sigma_k})(1+2C(\alpha_p)N_{p}^{\tau+\kappa_0}\|E^{p-1}\|_{\sigma_{p-1}})\nonumber\\
&\leq&8C_{\epsilon}\omega^2\epsilon C^3(\alpha_p)N_p^{3(\tau+\kappa_0)+4}\|E^{p-1}\|_{\sigma_{p-1}}^3(\sum_{k=0}^{p-1}\|w^{k}\|_{\sigma_k})\nonumber\\
&\leq&8C_{\epsilon}\omega^2\epsilon C^3(\alpha_p)N_p^{3(\tau+\kappa_0)+4}\|E^{p-1}\|_{\sigma_{p-1}}^3p\|w^{p}\|_{\sigma_{p}}\nonumber\\
&\leq&8C_{\epsilon}\omega^2\epsilon C^4(\alpha_p)N_{p}^{4(\tau+\kappa_0)+5}\|E^{p-1}\|_{\sigma_{p-1}}^4\nonumber\\
&\leq&(8C_{\epsilon}\omega^2\epsilon)^{4+1}N_{p}^{4(\tau+\kappa_0)+5}N_{p-1}^{4^2(\tau+\kappa_0)+5\times4}C^4(\alpha_p)C^{4^2}(\alpha_{p-1})\|E^{p-2}\|_{\sigma_{p-2}}^{4^2}\nonumber\\
&=&(8C_{\epsilon}\omega^2\epsilon)^{4+1}N_0^{4p(\tau+\kappa_0)+4^2(p-1)(\tau+\kappa_0)+5+5\times4}C^4(\alpha_p)C^{4^2}(\alpha_{p-1})\|E^{p-2}\|_{\sigma_{p-2}}^{4^2}\nonumber\\
&\leq&\cdots\nonumber\\
&\leq&(8C_{\epsilon}\omega^2\epsilon)^{\sum_{k=1}^{p-1}4^k+1}N_0^{(\tau+\kappa_0)4^{p+2}+5\times4^{p+1}}\|E^{0}\|_{\sigma_{0}}^{4^p}
\prod_{k=1}^{p}C^{4^k}(\alpha_{p+1-k})\nonumber\\
&\leq&(8C_{\epsilon}\omega^2\epsilon)^{4^p}(N_0^{16(\tau+\kappa_0)+20}\|E^{0}\|_{\sigma_{0}})^{4^p}
\prod_{k=1}^{p}C^{4^k}(\alpha_{p+1-k})\nonumber\\
&\leq&(8C_{\epsilon}\omega^2\epsilon)^{4^p}\|\tilde{E}^{0}\|_{\sigma_{0}+16(\tau+\kappa_0)+20}^{4^p}
\prod_{k=1}^{p}C^{4^k}(\alpha_{p+1-k})\nonumber\\
&\leq&(8^{4^2+1}C_{\epsilon}\omega^2\epsilon c^{16}_{\tau,\sigma,\bar{\sigma},\gamma_1,\gamma}\|\tilde{E}^0\|_{\sigma_{0}+16(\tau+\kappa_0)+20})^{4^p},
\end{eqnarray}
which means that if we choose a suitable small $\epsilon$ such that the second inequality in (\ref{E6-1}) holds for some $0<d<1$.
 Now we turn to estimate the term $\|w^p\|_{\sigma_p}$. It follows from (\ref{E3-18})--(\ref{E3-20}) that
\begin{eqnarray*}
\|w^p\|_{\sigma_p}
&\leq&C(\alpha_p)N_{p}^{\tau+\kappa_0}(8^{4^2+1}C_{\epsilon}\omega^2\epsilon c^{16}(\tau,\sigma,\bar{\sigma},\gamma_1,\gamma)\|\tilde{E}^0\|_{\sigma_{0}+16(\tau+\kappa_0)+20})^{4^{p-1}},
\end{eqnarray*}
which implies that the first inequality in (\ref{E6-1}) holds for a sufficient small $\epsilon$.

If
\begin{eqnarray*}
\epsilon cN_p^{4+\sigma_p}\sum_{k=0}^{p-1}\|w^{k}\|_{\sigma_{k}}\geq1,
\end{eqnarray*}
then by (\ref{E3-4}) in Lemma 2,  (\ref{E2-32}), (\ref{E2-32rr}) and (\ref{E6-2}), we derive
\begin{eqnarray}\label{E3-18R}
\|w^p\|_{\sigma_p}&=&\|-(\mathcal{J}_{\omega}^{(N_p)})^{-1}E^{p-1}\|_{\sigma_p}\nonumber\\
&\leq&C(\alpha_p)N_{p}^{\tau+\kappa_0}\|E^{p-1}\|_{\sigma_{p-1}}\left(1+\epsilon cN_p^{4+\sigma_p}\sum_{k=0}^{p-1}\|w^{k}\|_{\sigma_{k}}\right)^3\nonumber\\
&\leq&\epsilon^3c^3N_p^{12+3\sigma_p+\tau+\kappa_0}C(\alpha_p)\|E^{p-1}\|_{\sigma_{p-1}}(\sum_{k=0}^{p-1}\|w^{k}\|_{\sigma_{k}})^3\nonumber\\
&\leq&\epsilon^3c^3N_p^{12+3\sigma_p+\tau+\kappa_0}C(\alpha_p)p^3\|E^{p-1}\|_{\sigma_{p-1}}\|w^{p-1}\|_{\sigma_{p-1}}^3\nonumber\\
&\leq&\epsilon^3c^3N_p^{12+3\sigma_p+\tau+\kappa_0}C(\alpha_p)p^3\|E^{p-1}\|_{\sigma_{p-1}}.
\end{eqnarray}

In what follows, we need to divide into the following situations
\begin{eqnarray*}
&&Case~1^{*}:~\epsilon^3c^3N_p^{12+3\sigma_p+\tau+\kappa_0}C(\alpha_p)p^3\|E^{p-1}\|_{\sigma_{p-1}}>1,~\sum_{k=0}^{p-1}\|w^{k}\|_{\sigma_{k}}>C_{3,R'},\\
&&Case~2^{*}:~\epsilon^3c^3N_p^{12+3\sigma_p+\tau+\kappa_0}C(\alpha_p)p^3\|E^{p-1}\|_{\sigma_{p-1}}\leq1,~\sum_{k=0}^{p-1}\|w^{k}\|_{\sigma_{k}}>C_{3,R'},\\
&&Case~3^{*}:~\epsilon^3c^3N_p^{12+3\sigma_p+\tau+\kappa_0}C(\alpha_p)p^3\|E^{p-1}\|_{\sigma_{p-1}}>1,~\sum_{k=0}^{p-1}\|w^{k}\|_{\sigma_{k}}\leq C_{3,R'},\\
&&Case~4^{*}:~\epsilon^3c^3N_p^{12+3\sigma_p+\tau+\kappa_0}C(\alpha_p)p^3\|E^{p-1}\|_{\sigma_{p-1}}\leq1,~\sum_{k=0}^{p-1}\|w^{k}\|_{\sigma_{k}}\leq C_{3,R'}.
\end{eqnarray*}

Now we discuss the case $1^{*}$. The idea of proof of the rest three cases is the same, so we omit it.
By (\ref{E2-32r})--(\ref{E3-18}), we derive
\begin{eqnarray}\label{E3-20RRR}
\|E^p\|_{\sigma_p}&=&2\omega^2\epsilon\|\Psi^{(N_p)}(f(\sum_{k=0}^{p}w^k_t,\sum_{k=0}^{p}w^k_y,\sum_{k=0}^{p}w^k_{tt},\sum_{k=0}^{p}w^k_{yy})
-f(\sum_{k=0}^{p-1}w^k_t,\sum_{k=0}^{p-1}w^k_y,\sum_{k=0}^{p-1}w^k_{tt},\sum_{k=0}^{p-1}w^k_{yy})\nonumber\\
&&-D_wf(\sum_{k=0}^{p-1}w^k_t,\sum_{k=0}^{p-1}w^k_y,\sum_{k=0}^{p-1}w^k_{tt},\sum_{k=0}^{p-1}w^k_{yy})\|_{\sigma_p}\nonumber\\
&\leq&2\omega^2\epsilon C_{\epsilon}N_p^4\|w^p\|_{\sigma_p}^2(C_{3,R'}+\sum_{k=0}^{p-1}\|w^{k}\|_{\sigma_k})(1+\|w^p\|_{\sigma_p}^2)\nonumber\\
&\leq&4\omega^2\epsilon^{13} C_{\epsilon}N_p^{52+12\sigma+4\tau+4\kappa_0}C^4(\alpha_p)p^{12}\|E^{p-1}\|_{\sigma_{p-1}}^4\nonumber\\
&\leq&4\omega^2\epsilon^{13} C_{\epsilon}N_p^{64+12\sigma+4\tau+4\kappa_0}C^4(\alpha_p)\|E^{p-1}\|_{\sigma_{p-1}}^4\nonumber\\
&\leq&4^{4+1}\omega^{2+2\times4}\epsilon^{13+13\times4}C_{\epsilon}N_0^{(p+4(p-1))(64+12\sigma+4\tau+4\kappa_0)}C^4(\alpha_p)C^{4\times4}(\alpha_{p-1})\|E^{p-1}\|_{\sigma_{p-1}}^{4^2}\nonumber\\
&\leq&\ldots\nonumber\\
&\leq&(4\times16^{\tau}\omega^2\epsilon^{13}C_{\tau,\sigma,\bar{\sigma},\gamma_1,\gamma}N_0^{64+12\sigma+4\tau+4\kappa_0}\|E_0\|_{\sigma_0})^{4^{p+1}}.
\end{eqnarray}
We can choose a suitable small $\epsilon$ such that the second inequality in (\ref{E6-1}) holds for some $0<d<1$.
 Now we turn to estimate the term $\|w^p\|_{\sigma_p}$. It follows from (\ref{E3-18R})--(\ref{E3-20RRR}) that
\begin{eqnarray*}
\|w^p\|_{\sigma_p}
&\leq&\epsilon^3c^3N_p^{12+3\sigma_p+\tau+\kappa_0}C(\alpha_p)p^3(4\times16^{\tau}\omega^2\epsilon^{13}C_{\epsilon,\sigma,\bar{\sigma}}N_0^{64+12\sigma+4\tau+4\kappa_0}\|E_0\|_{\sigma_0})^{4^{p}},
\end{eqnarray*}
which implies that the first inequality in (\ref{E6-1}) holds for a sufficient small $\epsilon$.

Thus we conclude that (\ref{E6-1}) holds. Furthermore, we obtain that
\begin{eqnarray*}
\lim_{p\longrightarrow\infty}\|E^p\|_{\sigma_p}=0.
\end{eqnarray*}
By the decay rate of $\|E^p\|_{\sigma_p}$ and the relationship between $\|w^p\|_{\sigma_p}$ and $\|E^p\|_{\sigma_p}$, we get that the series $\sum_{k=0}^{\infty}w^k$ is convergence. Thus equation (\ref{E3-6}) has a solution
\begin{eqnarray*}
w^{\infty}:=\sum_{k=0}^{\infty}w^k\in
\textbf{H}_{\bar{\sigma}}.
\end{eqnarray*}
This completes the proof.
\end{proof}

 Next result gives the uniqueness of solutions for equation (\ref{E3-6}).
\begin{lemma}
Assume that $\omega$ satisfies (\ref{E1-4R}) and (\ref{E10-1}). Equation (\ref{E3-6})
has a unique solution $w\in
\textbf{H}_{\bar{\sigma}}\cap\textbf{B}_1(0)$ obtained in
Lemma 14.
\end{lemma}
\begin{proof}
Let $w,w'\in
\textbf{H}_{\bar{\sigma}}\cap\textbf{B}_1(0)$ be two solutions of
system (\ref{E3-6}), where
\begin{eqnarray*}
\textbf{B}_1(0):=\{w|\|w\|_s<\delta,~for~some~0<\delta<1,~\forall s>\sigma_0\}.
\end{eqnarray*}
Let $\psi=w-w'$. Our target is to prove $\psi=0$.
By (\ref{E3-6}), we have
\begin{eqnarray*}
\mathcal{J}_{\omega}\psi&+&2\omega^2\epsilon\Psi^{(N_p)}D_wf(w'_t,w'_y,w'_{tt},w'_{yy})\psi\nonumber\\
&&+2\omega^2\epsilon\Psi^{(N_p)}
(f(w_t,w_y,w_{tt},w_{yy})-f(w'_t,w'_y,w'_{tt},w'_{yy})-D_wf(w'_t,w'_y,w'_{tt},w'_{yy})\psi)=0,
\end{eqnarray*}
which implies that
\begin{eqnarray}
\label{E3-24}
\psi&=&-2\omega^2\epsilon(\mathcal{J}_{\omega}+2\omega^2\epsilon\Psi^{(N_p)}D_wf(w'_t,w'_y,w'_{tt},w'_{yy}))^{-1}\nonumber\\
&&\times\Psi^{(N_p)}(f(w_t,w_y,w_{tt},w_{yy})-f(w'_t,w'_y,w'_{tt},w'_{yy})-D_wf(w'_t,w'_y,w'_{tt},w'_{yy})\psi).
\end{eqnarray}
If
\begin{eqnarray*}
\epsilon cN_p^{4+\sigma}\|w'\|_{\sigma_{p-1}}<1,
\end{eqnarray*}
by (\ref{E3-3R3}), (\ref{E3-4}), $(\ref{E3-24})$ and $N_p=N_0^p$, $\forall p\in\textbf{N}$, we have
\begin{eqnarray}\label{E5-2}
\|\psi\|_{\sigma_p}&=&2\omega^2\epsilon\|(\mathcal{J}^{N_p}_{\omega})^{-1}\Psi^{(N_p)}(f(w_t,w_y,w_{tt},w_{yy})-f(w'_t,w'_y,w'_{tt},w'_{yy})-D_wf(w'_t,w'_y,w'_{tt},w'_{yy})\psi)\|_{\sigma_p}\nonumber\\
&\leq&2\omega^2\epsilon C_{\epsilon}C(\alpha_p)N_p^{\tau+\kappa_0+4}\|\psi\|_{\sigma_{p-1}}^2(1+\epsilon cN_p^{4+\sigma}\|w'\|_{\sigma_{p-1}})^3(C_{3,R'}+\|w'\|_{\sigma_{p-1}})(1+\|\psi\|_{\sigma_{p-1}}^2)\nonumber\\
&\leq&2\omega^2\epsilon C_{\epsilon}C(\alpha_p)N_p^{\tau+\kappa_0+4}\|\psi\|_{\sigma_{p-1}}^2\nonumber\\
&\leq&2^{1+2}\omega^{2+2\times2}\epsilon^{1+2} C_{\epsilon}C(\alpha_p)C^2(\alpha_{p-1})N_0^{(p+2(p-1))(\tau+\kappa_0+4)}\|\psi\|_{\sigma_{p-2}}^{2^2}\nonumber\\
&\leq&\cdots\nonumber\\
&\leq&(2^{\tau+1}\omega^{2}\epsilon C_{\tau,\sigma,\bar{\sigma},\gamma_1,\gamma}N_0^{\tau+\kappa_0+4}\|\psi\|_{\sigma_{0}})^{2^p}.
\end{eqnarray}
We choose a suitable $\epsilon$ such that
\begin{eqnarray*}
0<2^{\tau+1}\omega^{2}\epsilon C_{\tau,\sigma,\bar{\sigma},\gamma_1,\gamma}N_0^{\tau+\kappa_0+4}\|\psi\|_{\sigma_{0}}<1,
\end{eqnarray*}
and then let $p\longrightarrow\infty$, we obtain
\begin{eqnarray*}
\lim_{p\longrightarrow\infty}\|\psi\|_{\bar{\sigma}}=0.
\end{eqnarray*}
Using the same idea, we can deal with the case of $\epsilon cN_p^{4+\sigma}\|w'\|_{\sigma_{p-1}}\geq1$.
This ends the proof.
\end{proof}

Next result shows that the solution $w^{\infty}$ is a non-trivial solution for system (\ref{E3-6}).
\begin{lemma}
The periodic solution $w^{\infty}$ constructed in Lemma 14 is non-trivial.
\end{lemma}
\begin{proof}
By induction, we prove that the periodic solution $w_{\infty}$ is a non-trivial, i.e.,
\begin{eqnarray*}
w^{\infty}=\sum_{k=0}^{\infty}w^k\neq C,~~for~C\in\textbf{R}.
\end{eqnarray*}
As we chosen, $w^0$ is not constant. We prove that $w^0+w^1$ is not a constant, where $w^1$ is constructed in (\ref{E3-8}).
Arguing by contradiction, assume that it holds
\begin{eqnarray*}
w^0+w^1=C.
\end{eqnarray*}
By (\ref{E6-3}), we have
\begin{eqnarray*}
-(\mathcal{J}^{(N_1)}_{\omega})^{-1}E^0=C-w^0,
\end{eqnarray*}
which implies that
\begin{eqnarray}\label{E10-2}
E^0=-\mathcal{J}^{(N_1)}_{\omega}(C-w^0).
\end{eqnarray}
Note that $E^0=\mathcal{J}_{\omega}w^0+2\omega^2\epsilon\Psi^{(N_1)}f(w^0_t,w_y^0,w^0_{tt},w_{yy}^0)$.
By (\ref{E10-2}), we derive
\begin{eqnarray*}
\mathcal{J}_{\omega}w^0+2\omega^2\epsilon\Psi^{(N_1)}f(w^0_t,w_y^0,w^0_{tt},w_{yy}^0)=\mathcal{J}^{(N_1)}_{\omega}(w^0-C),
\end{eqnarray*}
which together with the definition of the linearized operator $\mathcal{J}^{(N_1)}_{\omega}$ in (\ref{E4-1}) gives that
\begin{eqnarray*}
f(w^0_t,w_y^0,w^0_{tt},w_{yy}^0)=\partial_wf(w^0_t,w_y^0,w^0_{tt},w_{yy}^0)w^0.
\end{eqnarray*}
But (\ref{E3-3R1}) tells us that above equality can not holds for $w^0\neq C$. This is because that if
above equality holds, then it follows from (\ref{E3-6}) that
\begin{eqnarray*}
\mathcal{J}_{\omega}w^0+2\omega^2\epsilon\Psi^{(N_0)}f(w^0_t,w_y^0,w^0_{tt},w_{yy}^0)=\mathcal{J}_{\omega}w^0+2\omega^2\epsilon\Psi^{(N_0)}\partial_wf(w^0_t,w_y^0,w^0_{tt},w_{yy}^0)w^0=0.
\end{eqnarray*}
Note that $\mathcal{J}^{(N_0)}_{\omega}=\mathcal{J}_{\omega}+2\omega^2\epsilon\Psi^{(N_0)}\partial_wf(w^0_t,w_y^0,w^0_{tt},w_{yy}^0)$ is an invertible operator.
This combining with (\ref{E3-3R1}) gives that $w^0=C$. This contradicts $w^0\neq C$ in Lemma 12.
Hence $w^0+w^1$ is not a constant.

Assume that $\sum_{a=0}^{k-1}w^{a}+w^k$ is not a constant. Then by the same process of above proof, we get that $\sum_{a=0}^kw^{a}+w^{k+1}$ is also not a constant.
For conclusion, we obtain the periodic solution $w_{\infty}=\sum_{k=0}^{\infty}w^k$ is non-trivial. This finishes the proof.
\end{proof}

%\begin{acknowledgements}
%The author is greatly indebted to Prof G. Tian for
%his discussion on Nash-Moser iteration scheme and this problem and much kind help!
%This work was done in Beijing International Center for Mathematical Research, Peking University. The author is supported by NSFC No 11201172, SRFDP Grant No 20120061120002.
%\end{acknowledgements}

\end{document}